\documentclass[11pt, english, twoside]{article}
\usepackage[margin=3cm]{geometry}
\usepackage{amsbsy,amsmath,amsthm,amssymb}

\usepackage{fourier}
\DeclareMathAlphabet{\mathcal}{OMS}{cmsy}{m}{n}

\usepackage[dvipsnames]{xcolor}
\usepackage{bm}
\usepackage{esint}
\usepackage{babel}
\usepackage[colorlinks=true, citecolor = blue]{hyperref}

\usepackage{indentfirst}
\usepackage{changepage, cases}	
\usepackage{setspace}
\usepackage{indentfirst}
\usepackage{graphicx}
\usepackage{calc}
\usepackage{orcidlink}
\usepackage{booktabs,multirow}
\usepackage{makecell}

\usepackage{fancyhdr}
\allowdisplaybreaks

\theoremstyle{plain}
\newtheorem{theorem}{Theorem}[section]
\newtheorem{lemma}[theorem]{Lemma}
\newtheorem{corollary}[theorem]{Corollary}
\newtheorem{proposition}[theorem]{Proposition}

\theoremstyle{remark}
\newtheorem{remark}[theorem]{Remark}

\theoremstyle{definition}
\newtheorem{definition}{Definition}

\newenvironment{proof of theorem 1.1}{{\noindent \em Proof of Theorem 1.1.}}{\hfill $\Box$\par}
\newenvironment{proof of theorem 1.2}{{\noindent \em Proof of Theorem 1.2.}}{\hfill $\Box$\par}

\DeclareSymbolFont{EulerExtension}{U}{euex}{m}{n}
\DeclareMathSymbol{\euintop}{\mathop} {EulerExtension}{"52}
\DeclareMathSymbol{\euointop}{\mathop} {EulerExtension}{"48}

\begin{document}
	\title{On a new class of $q$-fractional integral operators involving multivariate $q$-Mittag-Leffler functions}
	\author{
		Min-Jie Luo \orcidlink{0000-0001-7433-4490} $^{\rm 1}$\thanks{Corresponding author}~, 
		Qing-Hui Peng $^{\rm 2}$, Ravinder Krishna Raina $^{\rm 3}$} 
	\date{}
	\maketitle
	\begin{center}\small
		$^{1}$\emph{School of Mathematics and Statistics, Donghua University,\\ 
			Shanghai 201620, People's Republic of China.}\\
		E-mail: \url{mathwinnie@live.com}, \url{mathwinnie@dhu.edu.com}
	\end{center}
	\begin{center}\small
		$^{2}$\emph{School of Mathematics and Statistics, Donghua University,\\ 
			Shanghai 201620, People's Republic of China.}\\
		E-mail: \url{1213785099@qq.com}
	\end{center}
	\begin{center}\small
		$^{3}${\emph{M.P. University of Agriculture and Technology, Udaipur (Rajasthan), India\\
				\emph{Present address:} 10/11, Ganpati Vihar, Opposite Sector 5,\\
				Udaipur-313002, Rajasthan, India.}\\
			E-mail: \url{rkraina_7@hotmail.com}}
	\end{center}
	
	
	\begin{abstract}
		Fractional integral operators involving special functions as kernels have long been one of the main subjects of fractional calculus. As a discrete version of classical fractional calculus, $q$-fractional calculus not only brings often new phenomena and new perspectives but also provides new useful areas for describing and studying new problems of applied nature. Recently, Luo and Zhou (see below) developed a methodology for formulating well-behaved $q$-fractional integral operators steming from a basic general theory. Our aim in this paper is to define and introduce a new class of $q$-fractional integral operators whose kernel involves a new multivariate $q$-Mittag-Leffler function. We shall also define the corresponding $q$-fractional derivatives (as the inverses of $q$-fractional integral operators) and define also the $q$-Caputo fractional derivative operators. We discuss in detail the similarities and variations between the $q$-operators and their counterparts in classical theory, particularly in Prabakar fractional calculus. Further, applications of the defined operators in solving certain $q$-fractional differential equations are also considered. We also mention throughout this paper some useful relevances of our established results with certain known important cases. In the concluding remarks, we point out briefly few possible proposals (or directions) for further work on the present subject matter of $q$-calculus.
		
		\\
		
	\end{abstract}
	
	\tableofcontents 	
	
	\section{Introduction}\label{Introduction}
	
	\subsection{Background and motivation}\label{Introduction-1}

Originating from the works of Agarwal \cite{Agarwal-1969}, Al-Salam \cite{AlSalam-1966} and others, the theory of $q$-fractional calculus has after more than half a century of development become an independent subject and has attracted increasing attention from researchers in recent years (see \cite{AlTowailb-2019}, \cite{Annaby-Mansour-Book-2012}, \cite{Bouzeffour-2026},  \cite{Castro-2025}, \cite{Castro-2026},  \cite{Maligaranda-Oinarov-Persson-2014} and \cite{Zhang-Guo-2020}). 

The applications of $q$-fractional calculus to the problems of physical and applied sciences have been reported in few papers. We would like here to particularly mention one recent study. In 1988, Tsallis \cite{Tsallis-1988} proposed, as a generalization of Boltzmann--Gibbs statistics, the following entropy functional:
\begin{equation}\label{Def-TsalliE}
	S_q=k\frac{1-\sum_{i=1}^{W}p_i^q}{q-1},
\end{equation}
where $W$ is the total number of (microscopic) configurations, $\{p_i\}$ is the associated probabilities with $\sum_{i=1}^{W}p_i=1$, $k$ is a conventional positive constant, and $q\in\mathbb{R}$ is a nonextensive parameter. From the viewpoint of $q$-calculus, \eqref{Def-TsalliE} can be interpreted as the action of $q$-derivative operator $D_q$ (see below \eqref{Def-q-derivative}) acting on an entropy generating function, i.e.,  
\[
S_q=-k D_q\sum_{i=1}^{W}p_i^{x}\bigg|_{x=1}.
\]

In 2026, Gonzalez and Bayron \cite{Gonzalez-Bayron-2026} generalized the above idea by replacing $D_q$ with the $q$-Caputo fractional derivative ${}^{C}D_q^{\alpha}$ (see below \eqref{Def-CaputoQFD}), thereby obtaining a useful generalization of the Tsallis entropy functional with parameters $q$ and $\alpha$, in which 
$q$ still plays the role of the nonextensive index, while 
$\alpha$ provides a nonlocal memory-like contribution.

The main objects of study in $q$-fractional calculus are various $q$-fractional integral and derivative operators (including the $q$-Caputo fractional derivative) together with their defining spaces of functions and conditions of existence associated with them. Consequently , as in the classical fractional calculus, the main central problem is how to define useful $q$-fractional operators. Compared to the classical fractional calculus, the defining $q$-fractional calculus operators are still very few. This is perhaps due to the computational complexities arising from discrete objects and also due to a lack of a general framework for generating useful operators. In a recent work of Luo and Zhou \cite{Luo-Zhou-2026}, a general framework was, however, developed. Based on the summability of discrete convolutions of sequences as a criterion, this framework can explicitly suggest how to construct new well-behaved $q$-fractional operators. Therefore, the aim of this paper is to study a new class of $q$-fractional integral operators whose kernel involves a new multivariate $q$-Mittag-Leffler function as suggested by Luo--Zhou in their paper \cite{Luo-Zhou-2026}. We shall further define the corresponding $q$-fractional derivatives as their inverses and consider various results related to the $q$-calculus operators introduced in this paper. We also study some applications to the $q$-fractional differential equations. 

Before introducing our theory, let us first introduce some basic concepts from $q$-fractional calculus as well as some convenient notations used in the sequel. 

\subsection{Set of definitions of $q$-fractional calculus}\label{Introduction-2}

In the present investigation, we always assume that $0<q<1$. For any $a\in \mathbb{C}$, the $q$-shifted factorial $(a;q)_n$ is defined by
\begin{align}
	(a; q)_n:=
	\begin{cases} 
		1, & n = 0, \\
		\displaystyle \prod_{j=0}^{n-1} (1 - a q^j), & n\in \mathbb{Z}_{\geq1},
	\end{cases}
\end{align}
where $\mathbb{Z}_{\geq m}=\{n\in\mathbb{Z}; n\geq m\}$. Other notations such as $\mathbb{Z}_{\leq m}$ are understood analogously, and we also write 
\[
(a;q)_{\infty}:=\prod_{j=0}^{\infty}(1-aq^{j}).
\]
For any $\alpha\in\mathbb{C}$, we define
\[
(a;q)_\alpha:=\frac{(a;q)_\infty}{(a q^\alpha; q)_\infty},
\]
and for any $x \in \mathbb{C} \setminus \mathbb{Z}_{\leq 0}$, the $q$-gamma function $\Gamma_q(x)$ is defined by \cite[p. 22]{Annaby-Mansour-Book-2012}:
\begin{align}\label{Def-QGamma}
	\Gamma_q(x) := \frac{(q; q)_\infty}{(q^x; q)_\infty} (1-q)^{1-x}  = (q;q)_{x-1}(1-q)^{1-x}. 
\end{align}
It immediately follows from \eqref{Def-QGamma} that $\Gamma_q(1)=1$ and $\displaystyle \lim_{q \to 1^-} \Gamma_q(x) = \Gamma(x)$.

The $q$-derivative of a function $f$ is defined by (\cite[Chapter 1]{Kac-Cheung-Book-2002}; see also \cite[p. 27]{Gasper-Rahman-Book-2004})
\begin{equation}\label{Def-q-derivative}
	D_q f(x):= \frac{f(x)-f(qx)}{(1-q)x}, \quad x \neq 0.
\end{equation}
Obviously, $\displaystyle\lim_{q\rightarrow1^{-}}D_{q}f(x)=f'(x)$ if $f$ is differentiable at $x$. Higher-order $q$-derivatives $D_q^{n}$ are defined recursively by 
\[
D_q^n f(x):=D_q(D_q^{n-1}f(x)) \quad (n\in\mathbb{Z}_{\geq1})
\] 
with the convention that $D_q^0 f(x):= f(x)$. The $q$-integral of $f$ is defined by (\cite[p. 69]{Kac-Cheung-Book-2002}; see also \cite[p. 23]{Gasper-Rahman-Book-2004})
\begin{equation}\label{Def-q-integral}
	I_qf(x)\equiv \int_0^x f(t)\mathrm{d}_q t:= (1-q) \sum_{n=0}^\infty  xq^n f(x q^n)\quad (x>0).
\end{equation}
From \eqref{Def-q-derivative} and \eqref{Def-q-integral}, we have the relation: $D_q I_q f(x)=f(x)$.

If $f(x)$ is right-continuous at $0$, then $I_q D_q f(x)=f(x)-f(0)$. Here, the right-continuity condition at $0$ plays the role of the notion of being \emph{$q$-regular at zero} introduced in \cite[p. 7, Definition 1.3.2]{Annaby-Mansour-Book-2012}. If $f$ is Riemann integrable, then
\[
\lim_{q\rightarrow 1^{-}}I_{q}f(x)=\int_0^x f(t)\mathrm{d}t.
\]
Similarly, we define $I_q^n$ by
\[
I_q^nf(x):= I_q(I_q^{n-1}f(x))\quad (n\in\mathbb{Z}_{\geq1}), 
\]
where $I_q^0 f(x):= f(x)$.

The Riemann-Liouville $q$-fractional integral operator is defined by (\cite[Eqs. 
(4.24) and (4.26)]{Annaby-Mansour-Book-2012})
\begin{equation}\label{Def-qRLFIO}
	I_q^\lambda f(x)
	:=\frac{x^{\lambda-1}}{\Gamma_q(\lambda)} 
	\int_0^x (qt/x;q)_{\lambda-1} f(t)\mathrm{d}_q t
\end{equation}
with $I_q^0 f(x) = f(x)$. It may be pointed out here that there is another symbolic system in which the $q$-analogue of the power function $(t-s)^{\lambda}$ is denoted by $(t-s)_q^{\lambda}$, namely, $(x-t)_q^{\lambda}:=x^{\lambda}(t/x;q)_{\lambda}$.  Nevertheless, for the convenience of readers from various disciplines, we retain the $q$‑Pochhammer notation instead of 
$(x-t)_q^{\alpha}$. Note that if $\lambda>0$ and $f$ is Riemann integrable, the following limit relation holds \cite[p. 129, Lemma 4.12]{Annaby-Mansour-Book-2012}:
\begin{align}
	\lim_{q\to 1^-} I_q^\lambda f(x)=I^\lambda f(x),
\end{align}
where $I^\lambda$ denotes the classical Riemann-Liouville fractional integral operator \cite[Chapter 2]{Kilbas-Srivastava-Trujillo-Book-2006}. 

The application of the Babenko method (namely, the inverse operator method) to the $q$-Abel integral equation (see \cite[p. 2771]{Luo-Raina-Zhou-2025}; see also \cite{Luo-Zhou-2026}):
\[
f(x)+a I_q^\lambda f(x)=g(x)
\]
naturally gives rise to the $q$-Mittag-Leffler function defined by 
\begin{equation}\label{Def-qMittagLeffler-1}
	E_{\lambda,\rho}^{q}(\omega;z-c):=\sum_{n=0}^{\infty}\frac{\omega^n}{\Gamma_{q}(\lambda+\rho n)}\mathcal{E}_{z}^{-c}\left\{z^{\rho n}\right\}, 
\end{equation}
where $\mathcal{E}_{z}^{-c}$ is the $q$-translation operator such that $\mathcal{E}_{z}^{-c}\left\{z^{\rho n}\right\}=z^{\rho n}(c/z;q)_{\rho n}$. The function \eqref{Def-qMittagLeffler-1} could be considered as a $q$-analogue of the usual two-parametric Mittag-Leffler function, and this function has been used in \cite{Abdeljawad et al-2011}, \cite{Sheng-Zhang-2022a} and \cite{Zhang-Wang-2021}. The factor $z^{\rho n}(c/z;q)_{\rho n}$ is introduced to ensure compatibility with the operations of $q$-integration when \eqref{Def-qMittagLeffler-1} appears as an integral kernel. Compared to the classical Mittag‑Leffler function, the $q$-analogues often possess more possibilities of deeper investigations (see \cite[Chapter 7]{Annaby-Mansour-Book-2012}, \cite{Purohit-Kalla-2011} and \cite[p. 604, Eqs. (2.20) and (2.21)]{Bouzeffour-2026}). However, in the present paper the multivariate $q$-Mittag‑Leffler function which we study is a direct consequence of the defining relation \eqref{Def-qMittagLeffler-1}.

Throughout the paper, boldface letters always denote
vectors of dimension $k$; for instance $\mathbf{n}:=(n_1,\cdots,n_k)\in\mathbb{Z}_{\geq0}^k$ and
$\bm{\gamma}:=(\gamma_1,\cdots,\gamma_k)\in\mathbb{C}^{k}$. The inner product of two $k$-dimensional vectors $\mathbf{u}$ and $\mathbf{v}$ is defined by $\langle \mathbf{u}, \mathbf{v} \rangle:=u_1 v_1 +\cdots +u_k v_k$. In particular, we write $\langle\mathbf{n}\rangle:= n_1+\cdots+n_k$. In addition, 
\[
q^{\bm{\gamma}}~~~\text{means}~~~(q^{\gamma_1},\cdots,q^{\gamma_k}),
\]
\[
\mathbf{x}\circ\mathbf{y}~~~\text{means}~~~(x_1 y_1,\cdots, x_k y_k),
\]
and
\[
\sum_{\mathbf{n}=\mathbf{0}}^{\infty}\Omega(\mathbf{n})~~~\text{means}~~~
\sum_{n_1,\cdots,n_k=0}^{\infty}\Omega(n_1,\cdots,n_k).
\]

\subsection{Some definitions related to Luo--Zhu theory}\label{Introduction-3}

Let us start with the following multivariate $q$-Mittag-Leffler function which is derived from the Luo--Zhou theory \cite{Luo-Zhou-2026}.

\begin{definition}
	Let $\lambda\in\mathbb{C}$ $(\Re(\lambda)>0)$, and let $\bm{\gamma}, \bm{\rho}, \bm{\omega}\in\mathbb{C}^{k}$ $(\Re(\rho_j)>0, j=1,\cdots,k)$. The multivariate $q$-Mittag-Leffler function is defined as follows:  
	\begin{align}\label{Def-MultivariateQML}
		E_{\bm{\rho},\lambda}^{\bm{\gamma}}(q; 
		\bm{\omega}; \mathbf{z},\mathbf{c})
		&:=\sum_{\mathbf{n}=\mathbf{0}}^{\infty}
		\frac{(q^{\gamma_1};q)_{n_1}}{(q;q)_{n_1}}\cdots\frac{(q^{\gamma_k};q)_{n_k}}{(q;q)_{n_k}} 
		\frac{\omega_1^{n_1} \cdots \omega_k^{n_k}}{\Gamma_q(\lambda + \langle \bm{\rho}, \mathbf{n}\rangle)}z_1^{\rho_1 n_1}\big(c_1/z_1;q\big)_{\rho_1 n_1}\notag\\ 
		&\hspace{0.5cm}\cdot z_2^{\rho_2 n_2}\big(q^{\rho_1 n_1}c_2/z_2;q\big)_{\rho_2 n_2}
		\cdots z_k^{\rho_k n_k}\big(q^{\rho_1 n_1+\cdots+\rho_{k-1} n_{k-1}}c_k/z_k;q\big)_{\rho_k n_k},
	\end{align}
	where $0\leq c_j<z_j$ for $j=1,\cdots,k$. 
\end{definition}
\begin{remark}
	In terms of the $q$-translation operator $\mathcal{E}_{z}^{-c}$ defined in subsection \ref{Introduction-2}, the defining formula \eqref{Def-MultivariateQML} can be written in a more compact form: 
	\begin{align*}
		E_{\bm{\rho},\lambda}^{\bm{\gamma}}(q; 
		\bm{\omega}; \mathbf{z},\mathbf{c})
		&:=\sum_{\mathbf{n}=\mathbf{0}}^{\infty} 
		\frac{\omega_1^{n_1} \cdots \omega_k^{n_k}}{\Gamma_q(\lambda + \langle \bm{\rho}, \mathbf{n}\rangle)}
		\prod_{j=1}^{k}
		\frac{(q^{\gamma_j};q)_{n_j}}{(q;q)_{n_j}}
		\mathcal{E}_{z_j}^{-c_j q^{\rho_1 n_1+\cdots+\rho_{j-1}n_{j-1}}}\left\{z_j^{\rho_j n_j}\right\}.
	\end{align*}
	Thus, our function \eqref{Def-MultivariateQML} is obviously a multivariate generalization of \eqref{Def-qMittagLeffler-1}.
\end{remark}

By setting $\mathbf{c}=\mathbf{0}$ and $\bm{\omega}=\mathbf{1}$ in \eqref{Def-MultivariateQML}, and then letting $z_j\to z_j^{1/\rho_j}$ $(j=1,\cdots,k)$ in the resulting formula, we obtain
\begin{equation}\label{Def-MultivariateQML-1}
	E_{\bm{\rho},\lambda}^{\bm{\gamma}}(q;\mathbf{z})
	:=\sum_{\mathbf{n}=\mathbf{0}}^{\infty}
	\frac{(q^{\gamma_1};q)_{n_1}}{(q;q)_{n_1}}\cdots\frac{(q^{\gamma_k};q)_{n_k}}{(q;q)_{n_k}} 
	\frac{z_1^{n_1} \cdots z_k^{n_k}}{\Gamma_q(\lambda + \langle \bm{\rho}, \mathbf{n}\rangle)},
\end{equation}
which is evidently a $q$-analogue of the multivariate Mittag-Leffler function defined by (\cite[p. 536, Eq. (1.14)]{Saxena-Kalla-Saxena-2011})
\begin{equation}\label{Def-MultivariateML}
	E_{\bm{\rho},\lambda}^{\bm{\gamma}}(\mathbf{z})
	:=\sum_{\mathbf{n}=\mathbf{0}}^{\infty}
	\frac{(\gamma_1)_{n_1}}{n_1!}\cdots\frac{(\gamma_k)_{n_k}}{n_k!} 
	\frac{z_1^{n_1} \cdots z_k^{n_k}}{\Gamma(\lambda + \langle \bm{\rho}, \mathbf{n}\rangle)}.
\end{equation}
A detailed analysis of the function defined by \eqref{Def-MultivariateQML-1}, including its absolute convergence, will be given in Section \ref{Section-2}.

It is worth mentioning that the function \eqref{Def-MultivariateML} has been applied in some branches of mathematics. For example, Cinque \cite{Cinque-2021} uses it to represent the probability density of the sum of $N$ independent generalized Mittag-Leffler distributed random variables, and later, Cinque and Orsingher \cite{Cinque-Orsingher-2024} use this function in their study on certain types of fractional Cauchy problems with Dzherbashyan--Caputo derivatives. In addition, Sakhno and Storozhuk \cite{Sakhno-Storozhuk-2026} use this function in their study on the fractional Poisson process and the generalized fractional counting processes. 

The case when  $z_1=\cdots=z_k=z$ and $c_1=\cdots=c_k=c$ in \eqref{Def-MultivariateQML} produces the kernel of our new $q$-fractional integral operators. We denote it by
\begin{equation}\label{Def-MultivariateQML-2}
	E_{\bm{\rho},\lambda}^{\bm{\gamma}}(q; 
	\bm{\omega}; z,c)
	:=\sum_{\mathbf{n}=\mathbf{0}}^{\infty}
	\frac{(q^{\gamma_1};q)_{n_1}}{(q;q)_{n_1}}\cdots\frac{(q^{\gamma_k};q)_{n_k}}{(q;q)_{n_k}} 
	\frac{\omega_1^{n_1} \cdots \omega_k^{n_k}}{\Gamma_q(\lambda + \langle \bm{\rho}, \mathbf{n}\rangle)}z^{\langle\bm{\rho}, \mathbf{n}\rangle}(c/z;q)_{\langle\bm{\rho}, \mathbf{n}\rangle}.
\end{equation}

\begin{definition}
	Let $\lambda\in\mathbb{C}$ $(\Re(\lambda)>0)$, and let $\bm{\gamma}, \bm{\rho}, \bm{\omega}\in\mathbb{C}^{k}$ $(\Re(\rho_j)>0, j=1,\cdots,k)$. The $q$-fractional integral operator $E_{q;\bm{\rho},\lambda;(\bm{\omega})}^{\bm{\gamma}}$ is defined by 
	\begin{equation}\label{Def-OurOperator}
		\big(E_{q;\bm{\rho},\lambda;(\bm{\omega})}^{\bm{\gamma}}f\big)(x)
		:=\int_{0}^{x} x^{\lambda-1} (qt/x;q)_{\lambda-1} E_{\bm{\rho},\lambda}^{\bm{\gamma}}\big(q; \bm{\omega}; x, q^\lambda t\big)f(t)\mathrm{d}_q t,
	\end{equation}
	where $x>0$. 
\end{definition}

We first point out that the operator \eqref{Def-OurOperator} provides a $q$-analogue of the fractional integral operator \cite[p. 540, Eq. (4.1)]{Saxena-Kalla-Saxena-2011}:
\begin{equation} \label{Def-SaxenaOperator}
	\big(E_{\bm{\rho},\lambda;(\bm{\omega})}^{\bm{\gamma}}f\big)(x)
	:=\int_{0}^{x} (x-t)^{\lambda-1} E_{\bm{\rho},\lambda}^{\bm{\gamma}}
	\big[\omega_1(x-t)^{\rho_1},\cdots, \omega_k(x-t)^{\rho_k}\big]f(t) \mathrm{d}t,
\end{equation}
where $\lambda\in\mathbb{C}$ $(\Re(\lambda)>0)$ and  $\bm{\gamma},\bm{\rho},\bm{\omega}\in\mathbb{C}^k$ $(\Re(\rho_j)>0, j=1,\cdots,k)$. This operator contains the well-known Prabhakar fractional integral operator as a special case. The theory built around the Prabhakar fractional operators, namely \emph{Prabhakar fractional calculus} \cite{Giusti et al-2020} has become an important independent research topic in the field of fractional calculus.

In addition, our operator \eqref{Def-OurOperator} contains the following $q$-fractional integral operator of Prabhakar-type proposed in 2025  by Shaimardan \emph{et al.} \cite[p. 2008, Eq. (17)]{Shaimardan-Karimov-Ruzhansky-Mamanazarov-2025} as a special case: 
\begin{equation}\label{Def-ShaimardanOperator}
	(E_{q;\rho,\lambda;(\omega)}^\gamma f)(x)
	:=\int_0^x x^{\lambda-1}(qt/x;q)_{\lambda-1} E_{\rho,\lambda}^\gamma(q;\omega;x,q^\lambda t)f(t)\mathrm{d}_q t,
\end{equation}
where
\[
E_{\rho,\lambda}^{\gamma}(q; 
\omega; z,c):=\sum_{n=0}^{\infty}
\frac{(q^{\gamma};q)_{n}}{(q;q)_{n}} 
\frac{\omega^n}{\Gamma_q(\lambda + \rho n)}z^{\rho n}(c/z;q)_{\rho n},
\]
is a special case of \eqref{Def-MultivariateQML-2}. One may also refer to the paper \cite{Mamanazarov-2024} for a concise introduction \textcolor{blue}{to} the theory of \eqref{Def-ShaimardanOperator}. 

The paper is organized as follows. In Section \ref{Section-2}, we study various properties of our new multivariate $q$-Mittag-Leffler function $E_{\bm{\rho},\lambda}^{\bm{\gamma}}(q;\mathbf{z})$. In Section \ref{Section-3}, we characterize the operator $E_{q;\bm{\rho},\lambda;(\bm{\omega})}^{\bm{\gamma}}$ and construct its left and right inverses. Building on these results, we introduce the corresponding $q$-fractional derivatives. In Section \ref{Section-4}, we study the eigenvalue problems for the new $q$-fractional derivatives introduced in Section \ref{Section-3}. By using the $q$-Laplace transform, we obtain explicit solutions to these problems. Finally, in Section \ref{Section-5}, we outline some possible directions for future research.

\section{Properties of the multivariate $q$-Mittag-Leffler functions}\label{Section-2}

\subsection{Absolutely convergence}

Unlike the non-$q$ case, the multivariable $q$-Mittag-Leffler function defined by \eqref{Def-MultivariateQML-1} is \emph{not} an entire function. This is a very notable difference. The following theorem characterizes the domain of convergence of series \eqref{Def-MultivariateQML-1} in $\mathbb{C}^k$. 

\begin{theorem}\label{Th-Convergence} 
	Let $\lambda\in\mathbb{C}$, and let $\bm{\gamma}, \bm{\rho}\in\mathbb{C}^{k}$ $(\Re(\rho_j)>0, j=1,\cdots,k)$. The series in \eqref{Def-MultivariateQML-1} is absolutely convergent in the polydisck 
	\[
	\mathbb{D}:=\left\{\mathbf{z}\in\mathbb{C}^k; |z_j|<(1-q)^{-\Re(\rho_j)}, j=1,\cdots,k\right\}.
	\]
\end{theorem} 
\begin{proof}
	Since $0<q<1$, we have 
	\[
	\frac{(q^{\gamma_j};q)_{\infty}}{(q;q)_{\infty}}
	=\lim_{n_j\rightarrow+\infty} \frac{(q^{\gamma_j};q)_{n_j}}{(q;q)_{n_j}}=:C_j \quad (j=1,\cdots,k),
	\]
	and each $C_j$ satisfies $0\leq |C_j|<\infty$. Then there exists a constant $M_1>0$ such that for all $\mathbf{n}\in\mathbb{Z}_{\geq0}^{k}$:
	\[
	\left|\prod_{j=1}^{k}\frac{(q^{\gamma_j};q)_{n_j}}{(q;q)_{n_j}} \right| \leq M_1.
	\]
	Thus, it remains to check the convergence of multiple series
	\begin{equation}\label{Th-Convergence-Proof-1} 
		\sum_{\mathbf{n}=\mathbf{0}}^{\infty}
		\frac{|z_1|^{n_1} \cdots |z_k|^{n_k}}{|\Gamma_q(\lambda + \langle \bm{\rho}, \mathbf{n}\rangle)|}.
	\end{equation}
	
	In view of \eqref{Def-QGamma}, we have
	\[
	|\Gamma_q(\lambda + \langle \bm{\rho}, \mathbf{n}\rangle)|
	=\frac{(q;q)_{\infty}}{|(q^{\lambda + \langle \bm{\rho}, \mathbf{n}\rangle};q)_{\infty}|}(1-q)^{1-\Re(\lambda)-\Re(\langle \bm{\rho}, \mathbf{n}\rangle)},
	\]
	where
	\begin{align*}
		\big|(q^{\lambda + \langle \bm{\rho}, \mathbf{n}\rangle};q)_{\infty}\big|
		=\prod_{\ell=0}^{\infty}\big|1-q^{\lambda + \langle \bm{\rho}, \mathbf{n}\rangle+\ell}\big|
		&\leq \prod_{\ell=0}^{\infty}\big(1+q^{\Re(\lambda)+\Re(\langle\bm{\rho}, \mathbf{n}\rangle)+\ell}\big)\\
		&\leq \prod_{\ell=0}^{\infty}\big(1+q^{\Re(\lambda)+\ell}\big)
		=(-q^{\Re(\lambda)};q)_{\infty}.
	\end{align*}
	We obtain therefore the following useful inequality:
	\begin{equation}\label{Th-Convergence-Proof-2} 
		\frac{1}{|\Gamma_q(\lambda + \langle \bm{\rho}, \mathbf{n}\rangle)|}
		\leq M_2\prod_{j=1}^{k}(1-q)^{\Re(\rho_j)n_j},
	\end{equation}
	where $\lambda\in\mathbb{C}$, $\mathbf{n}\in\mathbb{Z}_{\geq0}^{k}$,  $\bm{\rho}\in\mathbb{C}^{k}$ $(\Re(\rho_j)>0, j=1,\cdots,k)$, and
	\[
	M_2:=\frac{(-q^{\Re(\lambda)};q)_{\infty}}{(q;q)_{\infty}}
	(1-q)^{\Re(\lambda)-1}.
	\]
	Consequently, we have then
	\[
	\sum_{\mathbf{n}=\mathbf{0}}^{\infty}
	\frac{|z_1|^{n_1} \cdots |z_k|^{n_k}}{|\Gamma_q(\lambda + \langle \bm{\rho}, \mathbf{n}\rangle)|}
	\leq M_2
	\prod_{j=1}^{k} 
	\left(\sum_{n_j=0}^{\infty}
	|z_j|^{n_j}(1-q)^{\Re(\rho_j) n_j}\right). 
	\]
	Thus the series \eqref{Th-Convergence-Proof-1} converges absolutely when each component of $\mathbf{z}$ satisfies $|z_j|<(1-q)^{-\Re(\rho_j)}$ $(j=1,\cdots,k)$. The proof is complete. 
\end{proof}
\begin{remark}
	~
	\begin{itemize}
		\item[(1)] As $q\rightarrow1^{-}$, we see that the domain of convergence mentioned in Theorem \ref{Th-Convergence} gradually expands to the whole $\mathbb{C}^k$. When $k=1$, the series \eqref{Def-MultivariateQML-1} reduces to 
		\[
		E_{\rho,\lambda}^{\gamma}(q;z)
		:=\sum_{n=0}^{\infty}
		\frac{(q^{\gamma};q)_{n}}{(q;q)_{n}}
		\frac{z^n}{\Gamma_q(\lambda + \rho n)},
		\]
		where $z\in\mathbb{C}$ satisfies $|z(1-q)^\rho|<1$. This convergence result is consistent with existing results in the literature \cite[p. 2006, Definition 2.1]{Shaimardan-Karimov-Ruzhansky-Mamanazarov-2025}. 
		\item[(2)] Letting $\bm{\rho}=\mathbf{1}=(1,\cdots,1)$ in \eqref{Def-MultivariateQML-1} and then using 
		$\Gamma_{q}(\lambda+\langle\mathbf{n}\rangle)=\Gamma_{q}(\lambda)
		(q^\lambda;q)_{\langle\mathbf{n}\rangle}(1-q)^{-\langle\mathbf{n}\rangle}$, we obtain
		\[
		E_{\mathbf{1},\lambda}^{\bm{\gamma}}\big(q;(1-q)^{-1}\mathbf{z}\big)
		=\frac{1}{\Gamma_q(\lambda)}\Phi_2^{(k)}\big[q^{\gamma_1},\cdots,q^{\gamma_k};q^{\lambda};q;\mathbf{z}\big], 
		\]
		where
		\[
		\Phi_2^{(k)}\big[q^{\gamma_1},\cdots,q^{\gamma_k};q^{\lambda};q;\mathbf{z}\big]:=
		\sum_{\mathbf{n}=\mathbf{0}}^{\infty}
		\frac{(q^{\gamma_1};q)_{n_1}\cdots(q^{\gamma_k};q)_{n_k}}{(q^\lambda;q)_{\langle\mathbf{n}\rangle}}
		\frac{z_1^{n_1}}{(q;q)_{n_1}}\cdots\frac{z_k^{n_k}}{(q;q)_{n_k}}
		\]
		is a $q$-analogue of Humbert's function \cite[p. 34, Eq. (8)]{Srivastava-Karlsson-Book-1985}
		\[
		\Phi_2^{(k)}\big[\gamma_1,\cdots,\gamma_k;\lambda;\mathbf{z}\big]:=
		\sum_{\mathbf{n}=\mathbf{0}}^{\infty}
		\frac{(\gamma_1)_{n_1}\cdots(\gamma_k)_{n_k}}{(\lambda)_{\langle\mathbf{n}\rangle}}
		\frac{z_1^{n_1}}{n_1!}\cdots\frac{z_k^{n_k}}{n_k!}\quad (\mathbf{z}\in\mathbb{C}^k).
		\]
		This theorem therefore determines that the domain of convergence of  $\Phi_2^{(k)}\big[q^{\gamma_1},\cdots,q^{\gamma_k};q^{\lambda};q;\mathbf{z}\big]$ is a polydisc $|z_1|<1,\cdots,|z_k|<1$ in $\mathbb{C}^k$. Ernst \cite[Definition 2.6]{Ernst-2023} considered the two-dimensional case of \eqref{Def-MultivariateQML-1}, which was denoted by $\Upsilon_2$ and it was stated that its domain of convergence is $|z_1|<\infty$, $|z_2|<\infty$, which does not appear to be a correct observation. To examine this, we consider the case in which $k=2$, $|z_1|>1$ and $z_2=0$ in \eqref{Def-MultivariateQML-1}. In this case, the series $q$-analogue of Humbert's function (mentioned above) reduces further to 
		\begin{equation}\label{Th-Convergence-Remark-1} 
			\sum_{n_1=0}^{\infty}
			\frac{(q^{\gamma_1};q)_{n_1}}{(q^\lambda;q)_{n_1}}
			\frac{z_1^{n_1}}{(q;q)_{n_1}}.
		\end{equation}
		Since 
		\[
		\frac{(q^{\gamma_1};q)_{n_1}}{(q^\lambda;q)_{n_1}}
		\frac{z_1^{n_1}}{(q;q)_{n_1}}\sim \frac{(q^{\gamma_1};q)_{\infty}}{(q^\lambda;q)_{\infty}(q;q)_{\infty}}z_1^{n_1}\quad (n_1\rightarrow+\infty),
		\]
		the series \eqref{Th-Convergence-Remark-1} is clearly divergent for $|z_1|>1$.
	\end{itemize}
\end{remark}

\begin{corollary}\label{Cor-Convergence} 
	Let $\lambda\in\mathbb{C}$, and let $\mathbf{c}, \bm{\gamma}, \bm{\rho}, \bm{\omega}\in\mathbb{C}^{k}$ $(\Re(\rho_j)>0, j=1,\cdots,k)$. The series in \eqref{Def-MultivariateQML} is absolutely convergent in the hyperrectangle
	\[
	\mathbb{D}_{\mathbf{c}}:=\left\{\mathbf{z}\in\mathbb{R}^k ;|c_j|< z_j<\frac{1}{1-q}|\omega_j|^{-1/\Re(\rho_j)},j=1,\cdots,k\right\}.
	\]
\end{corollary}
\begin{proof}
	Let $\alpha_j\equiv c_j/z_j$ with $|\alpha_j|<1$  $(j=1,\cdots,k)$. Note that 
	\begin{align*}
		&\big(\alpha_1;q\big)_{\rho_1 n_1}
		\big(\alpha_2q^{\rho_1 n_1};q\big)_{\rho_2 n_2}
		\cdots\big(\alpha_kq^{\rho_1 n_1+\cdots+\rho_{k-1} n_{k-1}};q\big)_{\rho_k n_k}\\
		&\hspace{0.5cm}=
		\frac{(\alpha_1;q)_{\infty}}{(\alpha_1 q^{\rho_1 n_1};q)_{\infty}}
		\frac{(\alpha_2 q^{\rho_1 n_1};q)_{\infty}}{(\alpha_2 q^{\rho_1 n_1+\rho_2 n_2};q)_{\infty}}
		\cdots
		\frac{(\alpha_k q^{\rho_1 n_1+\cdots+\rho_{k-1} n_{k-1}};q)_{\infty}}{(\alpha_k q^{\rho_1 n_1+\cdots+\rho_k n_k};q)_{\infty}}.
	\end{align*}
	Since $(|a|;q)_{\infty}\leq |(aq^{b};q)_{\infty}|\leq (-|a|;q)_{\infty}$ for $|a|<1$ and $\Re(b)\geq0$, we have 
	\begin{align*}
		&\left|\big(\alpha_1;q\big)_{\rho_1 n_1}
		\big(\alpha_2q^{\rho_1 n_1};q\big)_{\rho_2 n_2}
		\cdots\big(\alpha_k q^{\rho_1 n_1+\cdots+\rho_{k-1} n_{k-1}};q\big)_{\rho_k n_k}\right|\\
		&\hspace{0.5cm}\leq 
		\frac{(-|\alpha_1|;q)_{\infty}}{(|\alpha_1|;q)_{\infty}}
		\frac{(-|\alpha_2|;q)_{\infty}}{(|\alpha_2|;q)_{\infty}}
		\cdots
		\frac{(-|\alpha_k|;q)_{\infty}}{(|\alpha_k|;q)_{\infty}}=:M.
	\end{align*}
	Here, $M$ is a constant independent of $\mathbf{n}$. Thus, we only need to consider the convergence of the series
	\begin{align}\label{Cor-Convergence-Proof-1} 
		\sum_{\mathbf{n}=\mathbf{0}}^{\infty}
		\frac{r_1^{n_1} \cdots r_k^{n_k}}{|\Gamma_q(\lambda + \langle \bm{\rho}, \mathbf{n}\rangle)|},
	\end{align} 
	where $r_j=|\omega_j|z_j^{\Re(\rho_j)}$ $(j=1,\cdots,k)$. Making use of  the steps in the proof of Theorem \ref{Th-Convergence} to examine the series \eqref{Cor-Convergence-Proof-1}, the desired result of Corollary \ref{Cor-Convergence} thus follows immediately.
\end{proof}
According to Corollary \ref{Cor-Convergence}, we always assume in the definition \eqref{Def-OurOperator} of the operator  $E_{q;\bm{\rho},\lambda;(\bm{\omega})}^{\bm{\gamma}}$ that
\begin{equation}\label{Condition on x}
	0<x<R:=\frac{1}{1-q}\min_{1\leq j\leq k}\left\{|\omega_j|^{-1/\Re(\rho_j)}\right\}.
\end{equation}

\subsection{$q$-Laplace transform of the multivariate $q$-Mittag-Leffler functions}

The $q$-Laplace transform ${}_{q}L_{s}$ can be defined by (\cite[p. 30]{Annaby-Mansour-Book-2012})
\begin{align*}
	{}_{q}L_{s} f(x)
	&=\frac{1}{1-q}\int_{0}^{s^{-1}}E_q(-qsx)f(x)\mathrm{d}_q x\notag\\
	&=\frac{(q;q)_{\infty}}{s}\sum_{j=0}^{\infty}\frac{q^j}{(q;q)_j}f(s^{-1}q^{j})~~~(\Re(s)>0),
\end{align*}
where
\[
E_{q}(z):=\sum_{n=0}^{\infty}\frac{q^{n(n-1)/2}}{(q;q)_n}z^n=(-z;q)_{\infty}~~~(z\in\mathbb{C})
\]
is a $q$-exponential function \cite[p. 11]{Gasper-Rahman-Book-2004}. Since $\displaystyle \lim_{q\rightarrow 1^{-}}E_{q}(z(1-q))=\mathrm{e}^z$, we have
\begin{equation}
	\lim_{q\rightarrow1^{-}}(1-q)\cdot {}_{q}L_{s(1-q)}f(x)
	=\lim_{q\rightarrow1^{-}}\int_{0}^{s^{-1}(1-q)^{-1}}E_q(-qs(1-q)x)f(x)\mathrm{d}_q x
	=\mathcal{L}[f](s),
\end{equation}
where $\mathcal{L}$ denotes the usual Laplace transform. Regarding the basic properties of ${}_{q}L_{s}$, Annaby and Mansour gave a fairly comprehensive summary in their book \cite[Section 1.9]{Annaby-Mansour-Book-2012}.

\begin{proposition}\label{Prop-LaplaceT}
	Let $\lambda\in\mathbb{C}$ $(\Re(\lambda)>0)$,  $\bm{\gamma},\bm{\rho},\bm{\omega}\in\mathbb{C}^k$ $(\Re(\rho_j)>0, j=1,\cdots,k)$. Then, for the function defined by \eqref{Def-MultivariateQML-1}, we have
	\begin{align}\label{Prop-LaplaceT-1}
		&{}_{q}L_{s}z^{\lambda-1}E_{\bm{\rho},\lambda}^{\bm{\gamma}}(q;\omega_1 z^{\rho_1},\cdots,\omega_k z^{\rho_k})\notag\\
		&\hspace{0.5cm}=s^{-\lambda}(1-q)^{\lambda-1}\prod_{j=1}^{k}
		\frac{(q^{\gamma_j}\omega_j
			s^{-\rho_j}(1-q)^{\rho_j};q)_{\infty}}{(\omega_j
			s^{-\rho_j}(1-q)^{\rho_j};q)_{\infty}},
	\end{align}
	where $s\in\mathbb{C}$ satisfies $\Re(s)>0$ and 
	$|\omega_j|(1-q)^{\Re(\rho_j)}|s^{-\rho_j}|<1$ $(j=1,\cdots,k)$.
\end{proposition}
\begin{proof}
	Recall that (\cite[p. 31, Eq. (1.88)]{Annaby-Mansour-Book-2012})
	\[
	{}_{q}L_{s} z^{\lambda-1}=s^{-\lambda}(1-q)^{\lambda-1}\Gamma_q(\lambda)~~~(\Re(s)>0).
	\] 
	Then we have  
	\begin{align}\label{Prop-LaplaceT-Proof-1}
		&{}_{q}L_{s}z^{\lambda-1}E_{\bm{\rho},\lambda}^{\bm{\gamma}}(q;\omega_1 z^{\rho_1},\cdots,\omega_k z^{\rho_k})\notag\\
		&\hspace{0.5cm}=s^{-\lambda}(1-q)^{\lambda-1}\sum_{\mathbf{n}=\mathbf{0}}^{\infty}
		\frac{(q^{\gamma_1};q)_{n_1}}{(q;q)_{n_1}}\cdots\frac{(q^{\gamma_k};q)_{n_k}}{(q;q)_{n_k}} 
		\omega_1^{n_1}\cdots\omega_k^{n_k}
		s^{-\langle\bm{\rho},\mathbf{n}\rangle}(1-q)^{\langle\bm{\rho},\mathbf{n}\rangle}\notag\\
		&\hspace{0.5cm}=s^{-\lambda}(1-q)^{\lambda-1}
		\prod_{j=1}^{k}
		\left(\sum_{n_j=0}^{\infty}
		\frac{(q^{\gamma_j};q)_{n_j}}{(q;q)_{n_j}} 
		\big(\omega_j
		s^{-\rho_j}(1-q)^{\rho_j}\big)^{n_j}\right).
	\end{align}
	The result \eqref{Prop-LaplaceT-1} now follows by summing each single series in \eqref{Prop-LaplaceT-Proof-1} by using the $q$-binomial theorem \cite[p. 8, Eq. (1.3.2)]{Gasper-Rahman-Book-2004}:
	\begin{equation}\label{Th-qBinomial}
		\sum_{n=0}^{\infty}\frac{(a;q)_n}{(q;q)_n}z^n
		=\frac{(az;q)_{\infty}}{(z;q)_{\infty}}, \quad |z|<1.
	\end{equation}
\end{proof}
\begin{remark}
	If $\bm{\rho}\in\mathbb{R}_{>0}^{k}$, the condition $s\in\mathbb{C}$ in the statement of Proposition \ref{Prop-LaplaceT} can be replaced by a concise condition, namely, 
	\begin{equation}\label{Prop-LaplaceT-2}
		\Re(s)>0
		\quad\text{and}\quad
		|s|>(1-q)\max_{1\leq j\leq k}\left\{|\omega_j|^{1/\rho_j}\right\}.
	\end{equation}
\end{remark}

If we replace $s$ by $s(1-q)$ in \eqref{Prop-LaplaceT-1}, then we obtain 
\[
(1-q)\cdot {}_{q}L_{s}z^{\lambda-1}E_{\bm{\rho},\lambda}^{\bm{\gamma}}(q;\omega_1 z^{\rho_1},\cdots,\omega_k z^{\rho_k})
=s^{-\lambda}\prod_{j=1}^{k}
\frac{(q^{\gamma_j}\omega_j
	s^{-\rho_j};q)_{\infty}}{(\omega_j
	s^{-\rho_j};q)_{\infty}},
\]
which upon letting $q\rightarrow 1^{-}$ yields the following result \cite[p. 251, Lemma 2.4]{Ozarslan-Fernandez-2022}:
\[
\mathcal{L}\left[z^{\lambda-1}E_{\bm{\rho},\lambda}^{\bm{\gamma}}(\omega_1 z^{\rho_1},\cdots,\omega_k z^{\rho_k})\right](s)
=s^{-\lambda}\prod_{j=1}^{k}
\left(1-\frac{\omega_j}{s^{\rho_j}}\right)^{-\gamma_j},
\]
where $\Re(s)>0$, $\Re(\lambda)>0$, $\Re(\rho_j)>0$ $(j=1,\cdots,k)$.

\subsection{$q$-Fractional calculus of the multivariate $q$-Mittag-Leffler functions}

Let us start with the Riemann-Liouville $q$-fractional integral of the multivariate $q$-Mittag-Leffler functions $E_{\bm{\rho},\lambda}^{\bm{\gamma}}(q;\omega_1 x^{\rho_1},\cdots,\omega_k x^{\rho_k})$, where $0<x<R$ with $R$ defined in \eqref{Condition on x}.

\begin{proposition}
	Let $0<x<R$, $\alpha,\lambda\in\mathbb{C}$ $(\Re(\lambda)>0, \Re(\alpha)>0)$, and $\bm{\gamma},\bm{\rho},\bm{\omega}\in\mathbb{C}^k$ $(\Re(\rho_j)>0, j=1,\cdots,k)$. Then, for the function defined by \eqref{Def-MultivariateQML-1}, we have
	\[
	I_q^{\alpha}x^{\lambda-1}E_{\bm{\rho},\lambda}^{\bm{\gamma}}(q;\omega_1 x^{\rho_1},\cdots,\omega_k x^{\rho_k})
	=x^{\lambda+\alpha-1}E_{\bm{\rho},\lambda+\alpha}^{\bm{\gamma}}(q;\omega_1 x^{\rho_1},\cdots,\omega_k x^{\rho_k}),
	\]
	where $I_q^{\alpha}$ is the Riemann-Liouville $q$-fractional integral operator defined by \eqref{Def-qRLFIO}. 
\end{proposition}
\begin{proof}
	Since this result is a direct consequence of
	\[
	I_q^\alpha x^{\lambda-1}=\frac{\Gamma_q(\lambda)}{\Gamma_q(\lambda+\alpha)}x^{\lambda+\alpha-1},
	\]
	we choose to omit the details of the proof. 
\end{proof}

Next, let us consider the Riemann-Liouville $q$-fractional derivative of $E_{\bm{\rho},\lambda}^{\bm{\gamma}}(q;\omega_1 x^{\rho_1},\cdots,\omega_k x^{\rho_k})$. As usual, the Riemann-Liouville $q$-fractional derivative operator $D_q^{\alpha}$ $(\alpha>0)$ is defined by \cite[p. 124, Definition 4.3.3]{Annaby-Mansour-Book-2012}:
\begin{equation}\label{Def-qRLFDO}
	D_q^\alpha f(x):=D_q^k I_q^{k-\alpha}f(x),
\end{equation}
where $k=[\alpha]+1$ (here, $[\cdot]$ stands for the Gauss floor function). We have then 
\begin{equation}\label{q-RLFDofPowerF}
	D_q^\alpha x^{\lambda-1}=\frac{\Gamma_q(\lambda)}{\Gamma_q(\lambda-\alpha)}x^{\lambda-\alpha-1}.
\end{equation}
\begin{proposition}
	Let $0<x<R$, $\lambda>0$, $\alpha>0$, and  $\bm{\gamma},\bm{\rho},\bm{\omega}\in\mathbb{C}^k$ $(\Re(\rho_j)>0, j=1,\cdots,k)$. We have
	\[
	D_q^{\alpha}x^{\lambda-1}E_{\bm{\rho},\lambda}^{\bm{\gamma}}(q;\omega_1 x^{\rho_1},\cdots,\omega_k x^{\rho_k})
	=x^{\lambda-\alpha-1}E_{\bm{\rho},\lambda-\alpha}^{\bm{\gamma}}(q;\omega_1 x^{\rho_1},\cdots,\omega_k x^{\rho_k}),
	\]
	where $D_q^{\alpha}$ is defined by \eqref{Def-qRLFDO}. 
\end{proposition}
\begin{proof}
	The result follows immediately by using the formula \eqref{q-RLFDofPowerF}.
\end{proof}

\subsection{A single contour integral representation}

In this subsection, we would like to conclude the whole of Section \ref{Section-2} by discussing a simple contour integral. 
\begin{theorem}\label{Th-ContourIR}
	Let $\mathbf{z}\in\mathbb{D}$, $\lambda\in\mathbb{C}$, and let $\bm{\gamma}, \bm{\rho}\in\mathbb{C}^{k}$ $(\Re(\rho_j)>0, j=1,\cdots,k)$. Then, for the function defined by \eqref{Def-MultivariateQML-1}, we have
	\begin{equation}\label{Th-ContourIR-1}
		E_{\bm{\rho},\lambda}^{\bm{\gamma}}(q;\mathbf{z})
		=\frac{(1-q)^{\lambda-1}}{2\pi\mathrm{i}}
		\int_{\mathcal{H}}\frac{t^{-\lambda}}{(t;q)_{\infty}}\prod_{j=1}^{k}\frac{(q^{\gamma_j}z_j
			(1-q)^{\rho_j}t^{-\rho_j};q)_{\infty}}{(z_j
			(1-q)^{\rho_j}t^{-\rho_j};q)_{\infty}}\mathrm{d}t,
	\end{equation}
	where $\mathcal{H}$ is a Hankel contour encircled the disk: 
	\[
	|t|\leq \max_{1\leq j\leq k} |z_j|^{1/\Re(\rho_j)}\mathrm{e}^{|\Im(\rho_j)|\pi/\Re(\rho_j)}.
	\]
\end{theorem}
\begin{proof}
	Recall that 
	\[
	\frac{1}{\Gamma_q(x)}=\frac{(1-q)^{x-1}}{2\pi\mathrm{i}}\int_{\mathcal{H}}\frac{t^{-x}}{(t;q)_{\infty}}\mathrm{d}t\quad(\Re(x)>0),
	\]
	where $\mathcal{H}$ is a Hankel contour; see \cite[p. 212]{Annaby-Mansour-Book-2012}. 
	
	Using the series representation \eqref{Def-MultivariateQML-1} for $E_{\bm{\rho},\lambda}^{\bm{\gamma}}(q;\mathbf{z})$ and the $q$-binomial theorem \eqref{Th-qBinomial}, we have 
	\begin{align*}
		E_{\bm{\rho},\lambda}^{\bm{\gamma}}(q;\mathbf{z})
		&=\frac{(1-q)^{\lambda-1}}{2\pi\mathrm{i}}
		\int_{\mathcal{H}}\frac{t^{-\lambda}}{(t;q)_{\infty}}
		\prod_{j=1}^{k}
		\left(\sum_{n_j=0}^{\infty}
		\frac{(q^{\gamma_j};q)_{n_j}}{(q;q)_{n_j}}
		\big(z_j
		(1-q)^{\rho_j}t^{-\rho_j}\big)^{n_j}\right)\mathrm{d}t\\
		&=\frac{(1-q)^{\lambda-1}}{2\pi\mathrm{i}}
		\int_{\mathcal{H}}\frac{t^{-\lambda}}{(t;q)_{\infty}}\prod_{j=1}^{k}\frac{(q^{\gamma_j}z_j
			(1-q)^{\rho_j}t^{-\rho_j};q)_{\infty}}{(z_j
			(1-q)^{\rho_j}t^{-\rho_j};q)_{\infty}}\mathrm{d}t,
	\end{align*}
	where $\mathbf{z}\in\mathbb{D}$ and 
	\begin{equation}\label{Th-ContourIR-Proof-1}
		|z_j|(1-q)^{\Re(\rho_j)}|t^{-\rho_j}|<1\quad (j=1,\cdots,k).
	\end{equation}
	Note that \eqref{Th-ContourIR-Proof-1} can be simplified to $|z_j||t^{-\rho_j}|<1$ $(j=1,\cdots,k)$, we hence obtain \[
	|z_j|^{1/\Re(\rho_j)}\mathrm{e}^{|\Im(\rho_j)|\pi/\Re(\rho_j)}<|t|
	\quad (j=1,\cdots,k).
	\]
\end{proof}
\begin{remark}
	When $k=1$ in \eqref{Th-ContourIR-1}, we get
	\begin{equation}\label{Th-ContourIR-Remark-1}
		E_{\rho,\lambda}^{\gamma}(q;z)
		=\frac{(1-q)^{\lambda-1}}{2\pi\mathrm{i}}
		\int_{\mathcal{H}}\frac{t^{-\lambda}}{(t;q)_{\infty}}
		\frac{(q^{\gamma}z
			(1-q)^{\rho}t^{-\rho};q)_{\infty}}{(z
			(1-q)^{\rho}t^{-\rho};q)_{\infty}}\mathrm{d}t,
	\end{equation}
	where $|z|<(1-q)^{-\Re(\rho)}$, $\lambda,\gamma,\rho\in \mathbb{C}$ $(\Re(\rho)>0)$, and $|t|\leq |z|^{1/\Re(\rho)}\mathrm{e}^{|\Im(\rho)|\pi/\Re(\rho)}$. If we further set $\gamma=1$ in \eqref{Th-ContourIR-Remark-1} and assume that $\rho$ is a positive number, then we obtain the result of Annaby and Mansour \cite[p. 212, Theorem 7.8]{Annaby-Mansour-Book-2012}. Additionally, by noting that
	\[
	\lim_{q\rightarrow 1^{-}}
	\frac{1}{(t(1-q);q)_{\infty}}
	\frac{(q^{\gamma}z
		t^{-\rho};q)_{\infty}}{(z
		t^{-\rho};q)_{\infty}}=\frac{t^{\rho\gamma}\mathrm{e}^{t}}{(t^{\rho}-z)^{\gamma}},
	\]
	we can consider \eqref{Th-ContourIR-Remark-1} as the $q$-analogue of the contour integral representation for Prabhakar function \cite[p. 21]{Giusti et al-2020}.
\end{remark}

\section{Properties of integral operator \eqref{Def-OurOperator}}\label{Section-3}

\subsection{The $L_{q}^{p}$-boundedness}

For $1\leq p<\infty$ and $a>0$, we define the space $L_q^p\equiv L_q^p[0,a]$ by
\begin{equation}\label{Def-Lqp}
	L_q^p[0,a]:= \left\{f:[0,a]\rightarrow\mathbb{C}; \|f\|_p := \left( \int_0^a |f(t)|^p \, \mathrm{d}_q t \right)^{\frac{1}{p}} < \infty \right\}.
\end{equation}
It is well-known that $L_q^p[0,a]$ is a Banach space (see \cite[p. 15]{Annaby-Mansour-Book-2012} and \cite[p. 347, Theorem 3.8]{Cardoso-Petronilho-2015}). This space has already become one of the standard spaces used in studying various $q$-fractional integral operators (see \cite{Castro-2026}, \cite{Luo-Zhou-2026}, \cite{Shaimardan-Tokmagambetov-2021}, \cite{Shaimardan-Persson-Tokmagambetov-2020} and  \cite{Shaimardan-Karimov-Ruzhansky-Mamanazarov-2025}). It should be noted that some authors use a \emph{smaller} space, namely, the space obtained by replacing the norm $\|\cdot\|_p$ in definition \eqref{Def-Lqp} with the norm (see \cite[p. 15, Definition 1.4.1]{Annaby-Mansour-Book-2012})
\[
\|f\|_{p,*}:=\sup_{0<x\leq a}\left(\int_0^x|f(t)|^p \mathrm{d}_q t\right)^{\frac{1}{p}}.
\]
In the literature, this space is also denoted by similar notation; see, for example \cite{Zhang-Wang-2021}; see also \cite{Sheng-Zhang-2022a}. Obviously, $\|f\|_p\leq\|f\|_{p,*}$. To avoid confusion, we explicitly state here that in most of the cases the proofs that hold for $\|\cdot\|_{p,*}$ also hold true (with slight modification) for the norm defined in \eqref{Def-Lqp}.

Let us first state the boundedness result in $L_q^p$ for the familiar Riemann-Liouville $q$-fractional integral operator. Following the proofs in \cite{Annaby-Mansour-Book-2012} and \cite[p. 11]{Luo-Zhou-2026}, the following lemma can be established easily. 

\begin{lemma}\label{Lemma-LuoZhou}
	Let $\lambda\in\mathbb{C}$ $(\Re(\lambda)>0)$. The operator $I_q^\lambda$ defined by \eqref{Def-qRLFIO} is bounded on $L_q^p[0,a]$ for $1\leq p<\infty$, and satisfies
	\begin{equation}\label{Lemma-LuoZhou-1}
		\|I_q^{\lambda} f\|_p\leq \mathsf{k}_q(a;\lambda) \|f\|_p,
	\end{equation}
	where
	\begin{equation}
		\mathsf{k}_q(a;\lambda):= \frac{\Gamma_q(\Re(\lambda))}{|\Gamma_q(\lambda)|} \frac{a^{\Re(\lambda)}}{\Gamma_q(\Re(\lambda)+1)}.
	\end{equation}
\end{lemma}
\begin{remark}
	In addition, for $\lambda>0$, Shaimardan \emph{et al.} \cite[p. 4433]{Shaimardan-Persson-Tokmagambetov-2020} give a proof in $L_q^p[0,a]$ based on a different approach. 
\end{remark}

We now turn to the boundedness of the integral operator $E_{q;\bm{\rho},\lambda;(\bm{\omega})}^{\bm{\gamma}}$ defined by \eqref{Def-OurOperator}
on the space $L_q^p$. 

\begin{theorem}\label{Th-Boundedness}
	Let $\lambda\in\mathbb{C}$ $(\Re(\lambda)>0)$, $\bm{\gamma}, \bm{\rho}\in\mathbb{C}^k$ $(\Re(\rho_j)>0, j=1,\cdots,k)$. Also, let $\bm{\omega}\in\mathbb{C}^k$ satisfy 
	\begin{equation}\label{Th-Boundedness-1}
		a(1-q)|\omega_j|^{1/\Re(\rho_j)}<1\quad (j=1,\cdots,k).
	\end{equation}
	Then, the operator $E_{q;\bm{\rho},\lambda;(\bm{\omega})}^{\bm{\gamma}}$ defined by \eqref{Def-OurOperator} is bounded on  $L_q^p[0,a]$, that is
	\begin{equation}\label{Th-Boundedness-2}
		\|E_{q;\bm{\rho},\lambda;(\bm{\omega})}^{\bm{\gamma}} f\|_p \leq \mathfrak{C}\|f\|_p, \quad \forall f \in L_q^p[0,a],
	\end{equation}
	where
	\begin{equation}
		\mathfrak{C}:=a^{\Re(\lambda)}\frac{(1-q)^{\Re(\lambda)}}{1-q^{\Re(\lambda)}}
		\frac{(-q^{\Re(\lambda)};q)_{\infty}}{(q;q)_{\infty}}
		\prod_{j=1}^{k}\frac{(-q^{\Re(\gamma_j)}|\omega_j|a^{\Re(\rho_j)}(1-q)^{\Re(\rho_j)};q)_{\infty}}{(|\omega_j|a^{\Re(\rho_j)}(1-q)^{\Re(\rho_j)};q)_{\infty}}.
	\end{equation}
\end{theorem}
\begin{proof}
	First, we point out that the operator \eqref{Def-OurOperator} can also be expressed in terms of the Riemann-Liouville $q$-fractional integral operators as follows: 
	\begin{equation}\label{Th-Boundedness-Proof-1}
		(E_{q;\bm{\rho},\lambda;(\bm{\omega})}^{\bm{\gamma}} f)(x) 
		=\sum_{\mathbf{n}=\mathbf{0}}^{\infty} \frac{(q^{\gamma_1};q)_{n_1}}{(q;q)_{n_1}} 
		\cdots \frac{(q^{\gamma_k};q)_{n_k}}{(q;q)_{n_k}}
		\omega_1^{n_1}\cdots\omega_k^{n_k}
		\cdot 
		I_q^{\lambda + \langle \bm{\rho}, \mathbf{n}\rangle} f(x).
	\end{equation}
	Applying Lemma \ref{Lemma-LuoZhou}, we obtain
	\begin{equation}
		\|E_{q;\bm{\rho},\lambda;(\bm{\omega})}^{\bm{\gamma}} f\|_p
		\leq C\|f\|_p,
	\end{equation}
	where
	\begin{equation}\label{Th-Boundedness-Proof-2}
		C:=\sum_{\mathbf{n}=\mathbf{0}}^{\infty} 
		\frac{|(q^{\gamma_1};q)_{n_1}|}{(q;q)_{n_1}} 
		\cdots \frac{|(q^{\gamma_k};q)_{n_k}|}{(q;q)_{n_k}}
		|\omega_1|^{n_1}\cdots|\omega_k|^{n_k}
		\mathsf{k}_q(a;\lambda+\langle\bm{\rho},\mathbf{n}\rangle).
	\end{equation}
	
	Since $|(q^{\gamma_j};q)_{n_j}|\leq(-q^{\Re(\gamma_j)};q)_{n_j}$ $(j=1,\cdots,k)$ and
	\[
	\mathsf{k}_q(a;\lambda+\langle\bm{\rho},\mathbf{n}\rangle) 
	\leq \frac{1-q}{1-q^{\Re(\lambda)}}\frac{a^{\Re(\lambda) + \Re(\langle \bm{\rho},\mathbf{n}\rangle)}}{|\Gamma_q(\lambda + \langle\bm{\rho}, \mathbf{n}\rangle)|}, 
	\]
	we can see that
	\[
	C\leq \frac{1-q}{1-q^{\Re(\lambda)}}
	a^{\Re(\lambda)}
	\sum_{\mathbf{n}=\mathbf{0}}^{\infty} 
	\frac{(-q^{\Re(\gamma_1)};q)_{n_1}}{(q;q)_{n_1}} 
	\cdots \frac{(-q^{\Re(\gamma_k)};q)_{n_k}}{(q;q)_{n_k}}
	\frac{|\omega_1|^{n_1}\cdots|\omega_k|^{n_k}a^{\Re(\langle \bm{\rho},\mathbf{n}\rangle)}}{|\Gamma_q(\lambda + \langle\bm{\rho}, \mathbf{n}\rangle)|}.
	\]
	The use of \eqref{Th-Convergence-Proof-2} then gives
	\begin{align*}
		C&\leq \frac{1-q}{1-q^{\Re(\lambda)}}
		a^{\Re(\lambda)}\frac{(-q^{\Re(\lambda)};q)_{\infty}}{(q;q)_{\infty}}(1-q)^{\Re(\lambda)-1}\\
		&\hspace{1cm}\cdot 
		\prod_{j=1}^{k}\left(\sum_{n_j=0}^{\infty} 
		\frac{(-q^{\Re(\gamma_j)};q)_{n_j}}{(q;q)_{n_j}} 
		\big(|\omega_j|a^{\Re(\rho_j)}(1-q)^{\Re(\rho_j)}\big)^{n_j}\right).
	\end{align*}
	The desired result now follows by making use of the $q$-binomial theorem \eqref{Th-qBinomial} under the condition \eqref{Th-Boundedness-1}. In addition, the condition \eqref{Th-Boundedness-1} is equivalent to 
	\[
	a<R=\frac{1}{1-q}\min_{1\leq j\leq k}\left\{|\omega_j|^{-1/\Re(\rho_j)}\right\},
	\]
	which guarantees that the integral operator \eqref{Def-OurOperator} 
	is well defined on $[0,a]$. 
\end{proof}	
\begin{remark}
	~
	\begin{itemize}
		\item[(1)] When $p=1$, Theorem \ref{Th-Boundedness} can be regarded as a $q$-analogue of the $L^1$-boundedness result for $E_{\bm{\rho},\lambda;(\bm{\omega})}^{\bm{\gamma}}$ obtained by Saxena \emph{et al.} \cite[p. 541, Theorem 4.1]{Saxena-Kalla-Saxena-2011}.
		\item[(2)] Letting $k=1$ in Theorem \ref{Th-Boundedness}, we obtain
		\begin{equation}\label{Th-Boundedness-Remark-1}
			\|E_{q;\rho,\lambda;(\omega)}^{\gamma} f\|_p \leq \mathfrak{C}_1\|f\|_p, \quad \forall f \in L_q^p[0,a],
		\end{equation}
		where
		\begin{equation}
			\mathfrak{C}_1:=a^{\Re(\lambda)}\frac{(1-q)^{\Re(\lambda)}}{1-q^{\Re(\lambda)}}
			\frac{(-q^{\Re(\lambda)};q)_{\infty}}{(q;q)_{\infty}}
			\frac{(-q^{\Re(\gamma)}|\omega|a^{\Re(\rho)}(1-q)^{\Re(\rho)};q)_{\infty}}{(|\omega|a^{\Re(\rho)}(1-q)^{\Re(\rho)};q)_{\infty}}.
		\end{equation}
		Here, $E_{q;\rho,\lambda;(\omega)}^{\gamma}$ is defined by \eqref{Def-ShaimardanOperator}. 
		Thus, we reproduce the $L_q^1$-boundedness result for $E_{q;\rho,\lambda;(\omega)}^{\gamma}$ due to Shaimardan \emph{et al.} \cite[p. 2009, Proposition 3.4]{Shaimardan-Karimov-Ruzhansky-Mamanazarov-2025}. It also provides a $q$-analogue of the $L^p$-boundedness of the Prabhakar fractional integral operator \cite[p. 25, Eq. (5.3)]{Giusti et al-2020}. The authors are somewhat surprised to note that the $L^p$-boundedness of the Prabhakar operator \cite{Giusti et al-2020} in the non-$q$ case was posed as an incompletely resolved problem in a recently published paper of Ledesma \emph{et al.} \cite[p. 113]{Ledesma et al-2026}, and
		claimed further to have now been fully resolved in this
		very 2026 published paper.
	\end{itemize}
\end{remark}	

\subsection{The semigroup property}

One of the most important properties of a fractional integral is its semigroup property. The semigroup property of the Riemann-Liouville $q$-fractional integral operator $I_{q}^{\alpha}$ is given by (\cite[p. 4432, Lemma 2.3]{Shaimardan-Persson-Tokmagambetov-2020}; see also \cite[p. 505, Lemma 3]{Castro-2026})
\[
I_{q}^{\alpha}I_{q}^{\beta}f(x)
=I_{q}^{\beta}I_{q}^{\alpha}f(x)
=I_{q}^{\alpha+\beta}f(x), \quad f\in L_q^p[0,a]. 
\]

The following result follows from the Luo--Zhou theory (see \cite[p. 13, Theorem 4.4]{Luo-Zhou-2026} and \cite[p. 15, Corollary 4.8]{Luo-Zhou-2026}) (in view of the details given in Subsection \ref{Introduction-2}).
\begin{theorem}\label{Th-semigroup}
	Let $\lambda_1, \lambda_2 \in \mathbb{C}$ with $\Re(\lambda_1) > 0$ and $\Re(\lambda_2) > 0$. Then, 
	\begin{equation}\label{Th-semigroup-1}
		E_{q;\bm{\rho},\lambda_1;(\bm{\omega})}^{\bm{\alpha}} 
		\big( E_{q;\bm{\rho},\lambda_2;(\bm{\omega}\circ q^{\bm{\alpha}})}^{\bm{\beta}} f\big)(x) 
		=\big( E_{q;\bm{\rho},\lambda_1+\lambda_2;(\bm{\omega})}^{{\bm{\alpha}+\bm{\beta}}} f\big)(x),
	\end{equation}
	where $f\in L_q^p[0,a]$ and $\bm{\omega}\circ q^{\bm{\alpha}}=(\omega_1 q^{\alpha_1},\cdots, \omega_k q^{\alpha_k})$. In particular, we have
	\begin{equation}\label{Th-semigroup-2}
		E_{q;\bm{\rho},\lambda_1;(\bm{\omega})}^{-\bm{\gamma}} 
		\big(E_{q;\bm{\rho},\lambda_2;(\bm{\omega}\circ q^{-\bm{\gamma}})}^{\bm{\gamma}} f\big)(x) 
		=(I_{q}^{\lambda_1+\lambda_2}f)(x)
	\end{equation}
	and
	\begin{equation}\label{Th-semigroup-3}
		E_{q;\bm{\rho},\lambda_1;(\bm{\omega})}^{\bm{\gamma}} 
		\big(E_{q;\bm{\rho},\lambda_2;(\bm{\omega}\circ q^{\bm{\gamma}})}^{-\bm{\gamma}} f\big)(x) 
		=(I_{q}^{\lambda_1+\lambda_2}f)(x).
	\end{equation}
\end{theorem}
\begin{proof}
	The main idea of the proof originates from Theorem 4.4 and Corollary 4.8 in \cite{Luo-Zhou-2026}. However, for the sake of completeness and for the convenience of the reader, we provide a detailed proof.
	
	Using \eqref{Th-Boundedness-Proof-1}, we have
	\begin{align*}
		&E_{q;\bm{\rho},\lambda_1;(\bm{\omega})}^{\bm{\alpha}} 
		\big(E_{q;\bm{\rho},\lambda_2;(\bm{\omega}\circ q^{\bm{\alpha}})}^{\bm{\beta}} f\big)(x) 
		=\sum_{\mathbf{m}=\mathbf{0}}^{\infty}\sum_{\mathbf{n}=\mathbf{0}}^{\infty} \frac{(q^{\alpha_1};q)_{m_1}}{(q;q)_{m_1}} 
		\cdots \frac{(q^{\alpha_k};q)_{m_k}}{(q;q)_{m_k}}\\
		&\hspace{0.5cm}\cdot\frac{(q^{\beta_1};q)_{n_1}}{(q;q)_{n_1}} 
		\cdots \frac{(q^{\beta_k};q)_{n_k}}{(q;q)_{n_k}}
		\omega_1^{m_1+n_1}\cdots\omega_k^{m_k+n_k}
		q^{\langle\bm{\alpha},\mathbf{n}\rangle}
		I_q^{\lambda_1+\lambda_2+\langle \bm{\rho},\mathbf{m}+\mathbf{n}\rangle} f(x).
	\end{align*}
	Next, we rearrange the series to obtain:
	\begin{align*}
		&E_{q;\bm{\rho},\lambda_1;(\bm{\omega})}^{\bm{\alpha}} 
		\big(E_{q;\bm{\rho},\lambda_2;(\bm{\omega}\circ q^{\bm{\alpha}})}^{\bm{\beta}} f\big)(x)\\
		&\hspace{1cm}=\sum_{\mathbf{m}=\mathbf{0}}^{\infty}\left(\prod_{j=1}^{k}\sum_{n_j=0}^{m_j} \frac{(q^{\alpha_j};q)_{m_j-n_j}}{(q;q)_{m_j-n_j}} 
		\frac{(q^{\beta_j};q)_{n_j}}{(q;q)_{n_j}} q^{\alpha_j n_j}\right)
		\omega_1^{m_1}\cdots\omega_k^{m_k}
		I_q^{\lambda_1+\lambda_2+\langle \bm{\rho},\mathbf{m}\rangle} f(x).
	\end{align*}
	Using now the $q$-binomial theorem \cite[p. 25]{Gasper-Rahman-Book-2004}:
	\[
	\frac{(q^{\alpha+\beta};q)_m}{(q;q)_m}=\sum_{n=0}^{m}\frac{(q^\alpha;q)_{n}(q^\beta;q)_{m-n}}{(q;q)_{n}(q;q)_{m-n}}
	q^{n\beta}, 
	\]
	we get
	\begin{align*}
		&E_{q;\bm{\rho},\lambda_1;(\bm{\omega})}^{\bm{\alpha}} 
		\big(E_{q;\bm{\rho},\lambda_2;(\bm{\omega}\circ q^{\bm{\alpha}})}^{\bm{\beta}} f\big)(x)\\
		&\hspace{1cm}=\sum_{\mathbf{m}=\mathbf{0}}^{\infty}
		\frac{(q^{\alpha_1+\beta_1};q)_m}{(q;q)_m}
		\cdots \frac{(q^{\alpha_k+\beta_k};q)_m}{(q;q)_m}
		\omega_1^{m_1}\cdots\omega_k^{m_k}
		I_q^{\lambda_1+\lambda_2+\langle \bm{\rho},\mathbf{m}\rangle} f(x)\\
		&\hspace{1cm}=\big( E_{q;\bm{\rho},\lambda_1+\lambda_2;(\bm{\omega})}^{{\bm{\alpha}+\bm{\beta}}} f\big)(x).
	\end{align*}
	This completes the proof of the desired result \eqref{Th-semigroup-1} of Theorem \ref{Th-semigroup}.  
	
	Setting $\bm{\alpha}=-\bm{\gamma}$ and $\bm{\beta}=\bm{\gamma}$ in \eqref{Th-semigroup-1}, and noting that $(q^0;q)_n=0$ for $n\in\mathbb{Z}_{\geq1}$ and $(q^0;q)_{0}=1$, we readily obtain in view of \eqref{Def-OurOperator} (see also \eqref{Def-ShaimardanOperator}) the desired result \eqref{Th-semigroup-2}. By similar substitutions in \eqref{Th-semigroup-1}, we can easily obtain the other result \eqref{Th-semigroup-3}.	
\end{proof}
\begin{remark}
	~
	\begin{itemize}
		\item[(1)] Theorem \ref{Th-semigroup} is a $q$-analogue of Theorem 6.1 in \cite{Saxena-Kalla-Saxena-2011}. Compared with the non‑$q$ case, formula \eqref{Th-semigroup-1} is more complicated because it contains the factor $\bm{\omega}\circ q^{\bm{\alpha}}$. 
		\item[(2)] The special case that $k=1$ in \eqref{Th-semigroup-1} reduces to
		\[
		E_{q;\rho,\lambda_1;(\omega)}^{\alpha} 
		\big( E_{q;\rho,\lambda_2;(\omega q^{\alpha})}^{\beta} f\big)(x) 
		=\big( E_{q;\rho,\lambda_1+\lambda_2;(\omega)}^{\alpha+\beta} f\big)(x),
		\]
		which can also be derived from the result of Shaimardan \emph{et al.} \cite[p. 2008, Lemma 3.3]{Shaimardan-Karimov-Ruzhansky-Mamanazarov-2025}.
	\end{itemize}
\end{remark}

\begin{corollary}\label{Cor-semigroup-1}
	Let $\lambda>0$ and let $n = [\lambda]+1$ be a positive integer such that $n-1< \lambda\leq n$. Then, for any $f \in L_q^1[0,a]$, we have
	\begin{equation}\label{Cor-semigroup-1-1}
		D_q^n\big(E_{q;\bm{\rho},n-\lambda;(\bm{\omega})}^{-\bm{\gamma}} 
		\big(E_{q;\bm{\rho},\lambda;(\bm{\omega}\circ q^{-\bm{\gamma}})}^{\bm{\gamma}} f\big)\big)(x) 
		=f(x).
	\end{equation}
\end{corollary}
\begin{proof}
	Let $n-1<\lambda\leq n$, and let $\lambda_1=n-\lambda$ and $\lambda_2=\lambda$ in  Eq. \eqref{Th-semigroup-2} of Theorem \ref{Th-semigroup}. We obtain then
	\begin{equation}\label{Cor-semigroup-Proof-1}
		E_{q;\bm{\rho},n-\lambda;(\bm{\omega})}^{-\bm{\gamma}} 
		\big(E_{q;\bm{\rho},\lambda;(\bm{\omega}\circ q^{-\bm{\gamma}})}^{\bm{\gamma}} f\big)(x) 
		=(I_{q}^n f)(x).
	\end{equation}
	Applying $D_q^n$ on both sides of \eqref{Cor-semigroup-Proof-1} gives immediately 
	\[
	D_q^n\big(E_{q;\bm{\rho},n-\lambda;(\bm{\omega})}^{-\bm{\gamma}} 
	\big(E_{q;\bm{\rho},\lambda;(\bm{\omega}\circ q^{-\bm{\gamma}})}^{\bm{\gamma}} f\big)\big)(x) 
	=D_q^n (I_{q}^n f)(x)=f(x).
	\]
\end{proof}

A function $f:[0,a]\rightarrow\mathbb{C}$ is called \emph{$q$-absolutely continuous} if there exists a $\varphi \in L^1_q[0,a]$ such that 
\[
f(x)=f(0)+\int_0^x \varphi(t)\mathrm{d}_q t
\]
for all $x\in[0,a]$ (see \cite[p. 4433, Definition 2.4]{Shaimardan-Persson-Tokmagambetov-2020}). Let $AC_q[0,a]$ denote the set of all $q$-absolutely continuous functions on $[0,a]$. Also, let 
\[
AC_q^n[0,a]:=\left\{f:[0,a]\rightarrow\mathbb{C};\begin{array}{l}
	\text{(i)}~~(D_q^{n-1}f)(x) \in AC_q[0,a], \\
	\text{(ii)}~~\text{$D_q^{j}f(x)$ $(j=0,\cdots,n-1)$ are right-continuous at $0$}
\end{array}\right\}.
\]
Obviously, if $f\in AC_q^n[0,a]$, then $D_q^n f\in L_q^1[0,a]$.

\begin{corollary}\label{Cor-semigroup-2}
	Let $n-1<\lambda\leq n$. If $f\in AC_{q}^{n}[0,a]$, then we have
	\begin{equation}\label{Cor-semigroup-2-1}
		E_{q;\bm{\rho},\lambda;(\bm{\omega})}^{\bm{\gamma}} 
		\big(E_{q;\bm{\rho},n-\lambda;(\bm{\omega}\circ q^{\bm{\gamma}})}^{-\bm{\gamma}} (D_q^n f)\big)(x)  =f(x)-\sum_{j=0}^{n-1}\frac{(D_q^j f)(0)}{\Gamma_q(j+1)} x^j.
	\end{equation}
\end{corollary}
\begin{proof} 
	Let $n-1<\lambda\leq n$, and let $\lambda_1=\lambda$ and $\lambda_2=n-\lambda$ in \eqref{Th-semigroup-3}. We obtain
	\begin{equation}\label{Cor-semigroup-2-Proof-1}
		E_{q;\bm{\rho},\lambda;(\bm{\omega})}^{\bm{\gamma}} 
		\big(E_{q;\bm{\rho},n-\lambda;(\bm{\omega}\circ q^{\bm{\gamma}})}^{-\bm{\gamma}} f\big)(x) 
		=(I_{q}^n f)(x).
	\end{equation}
	Replacing $f(x)$ by $D_q^n f(x)\in L_q^1[0,a]$ in \eqref{Cor-semigroup-2-Proof-1}, we obtain 
	\begin{equation}\label{Cor-semigroup-2-Proof-2}
		E_{q;\bm{\rho},\lambda;(\bm{\omega})}^{\bm{\gamma}} 
		\big(E_{q;\bm{\rho},n-\lambda;(\bm{\omega}\circ q^{\bm{\gamma}})}^{-\bm{\gamma}} (D_q^n f)\big)(x) 
		=(I_{q}^n (D_q^n f))(x).
	\end{equation}
	By recalling the $q$-Taylor formula \cite[p. 301, Eq. (10)]{Stankovc-Rajkovic-Marinkovic-2009}
	\begin{align}\label{Cor-semigroup-2-Proof-3}
		I_q^n (D_q^n f)(x)=f(x)-\sum_{j=0}^{n-1}\frac{(D_q^j f)(0^+)}{\Gamma_q(j+1)} x^j,
	\end{align}
	where $D_q^n f\in L_q^1[0,a]$ and $(D_q^{j}f)(0^+)$ $(j=0,\cdots,n-1)$ exists, then \eqref{Cor-semigroup-2-Proof-2} in conjuction with \eqref{Cor-semigroup-2-Proof-3} gives the desired result \eqref{Cor-semigroup-2-1}. This completes the proof.
\end{proof}

\subsection{$q$-Laplace transforms}

Here, we study the $q$-Laplace transform of the operator $E_{q;\bm{\rho},\lambda;(\bm{\omega})}^{\bm{\gamma}}$ and provide a new insight into the semigroup property. 

\begin{proposition}\label{Prop-LaplaceT-O}
	Let $\lambda\in\mathbb{C}$ $(\Re(\lambda)>0)$,  $\bm{\gamma},\bm{\rho},\bm{\omega}\in\mathbb{C}^k$ $(\Re(\rho_j)>0, j=1,\cdots,k)$. We have
	\begin{equation}\label{Prop-LaplaceT-O-1}
		{}_{q}L_{s}\big(E_{q;\bm{\rho},\lambda;(\bm{\omega})}^{\bm{\gamma}} f\big)(x)
		=s^{-\lambda}(1-q)^{\lambda}{}_q L_{s}f(x)\cdot
		\prod_{j=1}^{k}\frac{(q^{\gamma_j}\omega_j(1-q)^{\rho_j}
			s^{-\rho_j};q)_{\infty}}{(\omega_j(1-q)^{\rho_j}
			s^{-\rho_j};q)_{\infty}},
	\end{equation}
	where $s\in\mathbb{C}$ satisfies $\Re(s)>0$ and 
	$|\omega_j|(1-q)^{\Re(\rho_j)}|s^{-\rho_j}|<1$ $(j=1,\cdots,k)$.
\end{proposition}
\begin{proof}
	Recall that if $f\in L_q^1$, we have (\cite[p. 271, Eq. (9.1)]{Annaby-Mansour-Book-2012})
	\begin{equation}\label{Prop-LaplaceT-O-Proof-1}
		{}_{q}L_{s} (I_q^{\alpha}f)(x)=s^{-\alpha}(1-q)^{\alpha} {}_q L_{s}f(x)\quad (\Re(\alpha)>0).
	\end{equation}
	Applying \eqref{Prop-LaplaceT-O-Proof-1} to \eqref{Th-Boundedness-Proof-1} yields 
	\begin{align}
		{}_{q}L_{s}\big(E_{q;\bm{\rho},\lambda;(\bm{\omega})}^{\bm{\gamma}} f\big)(x) 
		&=s^{-\lambda}(1-q)^{\lambda}{}_q L_{s}f(x)\notag\\
		&\hspace{1cm}\cdot \sum_{\mathbf{n}=\mathbf{0}}^{\infty} \frac{(q^{\gamma_1};q)_{n_1}}{(q;q)_{n_1}} 
		\cdots \frac{(q^{\gamma_k};q)_{n_k}}{(q;q)_{n_k}}
		\omega_1^{n_1}\cdots\omega_k^{n_k}
		\cdot 
		s^{-\langle \bm{\rho}, \mathbf{n}\rangle}
		(1-q)^{\langle \bm{\rho}, \mathbf{n}\rangle}\notag\\
		&=s^{-\lambda}(1-q)^{\lambda}{}_q L_{s}f(x)\notag\\
		&\hspace{1cm}\cdot \prod_{j=1}^{k}\left(\sum_{n_j=0}^{\infty} \frac{(q^{\gamma_j};q)_{n_j}}{(q;q)_{n_j}} 
		\left(\omega_j(1-q)^{\rho_j}
		s^{-\rho_j}\right)^{n_j}
		\right)\notag\\
		&=s^{-\lambda}(1-q)^{\lambda}{}_q L_{s}f(x)\cdot
		\prod_{j=1}^{k}\frac{(q^{\gamma_j}\omega_j(1-q)^{\rho_j}
			s^{-\rho_j};q)_{\infty}}{(\omega_j(1-q)^{\rho_j}
			s^{-\rho_j};q)_{\infty}},
	\end{align}
	where $|\omega_j|(1-q)^{\Re(\rho_j)}|s^{-\rho_j}|<1$ $(j=1,\cdots,k)$. 
\end{proof}
\begin{remark}
	Using Proposition \ref{Prop-LaplaceT}, we can express the product of quotients of $q$-Pochhammer symbols as the $q$-Laplace transform of the multivariate $q$-Mittag-Leffler function as follows: 
	\begin{equation}\label{Prop-LaplaceT-O-Remark-1}
		{}_{q}L_{s}\big(E_{q;\bm{\rho},\lambda;(\bm{\omega})}^{\bm{\gamma}} f\big)(x)
		=(1-q)\cdot {}_q L_{s}f(x)\cdot {}_{q}L_{s}z^{\lambda-1}E_{\bm{\rho},\lambda}^{\bm{\gamma}}(q;\omega_1 z^{\rho_1},\cdots,\omega_k z^{\rho_k}).
	\end{equation}
\end{remark}

From the viewpoint of $q$-Laplace transform, we can provide an \emph{alternative derivation} of the semigroup property \eqref{Th-semigroup-1}. Actually, we have
\begin{align}\label{Prop-LaplaceT-O-Remark-2}
	&{}_{q}L_{s}\big(E_{q;\bm{\rho},\lambda_1;(\bm{\omega})}^{\bm{\alpha}} 
	\big( E_{q;\bm{\rho},\lambda_2;(\bm{\omega}\circ q^{\bm{\alpha}})}^{\bm{\beta}} f\big)\big)(x)\notag\\
	&\hspace{0.5cm}={}_q L_{s}\big( E_{q;\bm{\rho},\lambda_2;(\bm{\omega}\circ q^{\bm{\alpha}})}^{\bm{\beta}} f\big)(x)\cdot s^{-\lambda_1}(1-q)^{\lambda_1}\prod_{j=1}^{k}\frac{(q^{\alpha_j}\omega_j(1-q)^{\rho_j}
		s^{-\rho_j};q)_{\infty}}{(\omega_j(1-q)^{\rho_j}
		s^{-\rho_j};q)_{\infty}},
\end{align}
where 
\begin{equation}\label{Prop-LaplaceT-O-Remark-3}
	{}_q L_{s}\big( E_{q;\bm{\rho},\lambda_2;(\bm{\omega}\circ q^{\bm{\alpha}})}^{\bm{\beta}} f\big)(x)
	={}_q L_{s}f(x)\cdot s^{-\lambda_2}(1-q)^{\lambda_2}
	\prod_{j=1}^{k}\frac{(q^{\alpha_j+\beta_j}\omega_j(1-q)^{\rho_j}
		s^{-\rho_j};q)_{\infty}}{(q^{\alpha_j}\omega_j(1-q)^{\rho_j}
		s^{-\rho_j};q)_{\infty}}.
\end{equation}
Combining \eqref{Prop-LaplaceT-O-Remark-2} and \eqref{Prop-LaplaceT-O-Remark-3}, we obtain
\begin{align*}
	&{}_{q}L_{s}\big(E_{q;\bm{\rho},\lambda_1;(\bm{\omega})}^{\bm{\alpha}} 
	\big( E_{q;\bm{\rho},\lambda_2;(\bm{\omega}\circ q^{\bm{\alpha}})}^{\bm{\beta}} f\big)\big)(x)\\
	&\hspace{0.5cm}={}_q L_{s}f(x)\cdot s^{-\lambda_1-\lambda_2}(1-q)^{\lambda_1+\lambda_2}
	\prod_{j=1}^{k}\frac{(q^{\alpha_j+\beta_j}\omega_j(1-q)^{\rho_j}
		s^{-\rho_j};q)_{\infty}}{(\omega_j(1-q)^{\rho_j}
		s^{-\rho_j};q)_{\infty}}\\
	&\hspace{0.5cm}={}_{q}L_{s}(E_{q;\bm{\rho},\lambda_1+\lambda_2;(\bm{\omega})}^{\bm{\alpha}+\bm{\beta}} f)(x).
\end{align*}

\subsection{Construction of the inverse operators}

Inspired by the Corollary \ref{Cor-semigroup-1}, we can construct a right inverse for the operator $E_{q;\bm{\rho},\lambda;(\bm{\omega})}^{\bm{\gamma}}$. 

\begin{definition}\label{Def-RLtype}
	Let $f\in L_q^1[0,a]$, so that $E_{q;\bm{\rho},n-\lambda;(\bm{\omega})}^{-\bm{\gamma}} f\in AC_q^n[0,a]$. The Riemann-Liouville-type $q$-fractional derivative $\mathcal{D}_{q;\bm{\rho},\lambda;(\bm{\omega})}^{\bm{\gamma}}$ is defined by
	\begin{equation}\label{Def-RLtype-1}
		(\mathcal{D}_{q;\bm{\rho},\lambda;(\bm{\omega})}^{\bm{\gamma}} f)(x)
		=D_q^n\big(E_{q;\bm{\rho},n-\lambda;(\bm{\omega})}^{-\bm{\gamma}} f\big)(x),
	\end{equation}
	where $n=[\lambda]+1$.
\end{definition}
\begin{remark}\label{Def-RLtype-Remark}
	~
	\begin{itemize}
		\item[(1)] In addition to Corollary \ref{Cor-semigroup-1}, we can also justify Definition \ref{Def-RLtype} from the perspective of integral equations. Consider the $q$-integral equation
		\[
		(E_{q;\bm{\rho},\lambda;(\bm{\omega}\circ q^{-\bm{\gamma}})}^{\bm{\gamma}}f)(x)=g(x),
		\]
		where $g$ is a known function. Applying the $q$-Laplace transform on both sides of the above equation, and taking into account Propositions \ref{Prop-LaplaceT-O} and \ref{Prop-LaplaceT}, we obtain 
		\begin{align*}
			{}_{q}L_s f(x)
			&={}_{q}L_s g(x)\cdot\left(\frac{s}{1-q}\right)^n\cdot s^{-(n-\lambda)}(1-q)^{n-\lambda}\prod_{j=1}^{k}\frac{(q^{-\gamma_j}\omega_j(1-q)^{\rho_j}
				s^{-\rho_j};q)_{\infty}}{(\omega_j(1-q)^{\rho_j}
				s^{-\rho_j};q)_{\infty}}\\
			&=(1-q)\cdot\left(\frac{s}{1-q}\right)^n\cdot {}_{q}L_s g(x)\cdot
			{}_{q}L_{s}z^{n-\lambda-1}E_{\bm{\rho},n-\lambda}^{-\bm{\gamma}}(q;\omega_1 z^{\rho_1},\cdots,\omega_k z^{\rho_k})\\
			&=\left(\frac{s}{1-q}\right)^n\cdot {}_{q}L_{s}(E_{q;\bm{\rho},n-\lambda;(\bm{\omega})}^{-\bm{\gamma}} g)(x).
		\end{align*}
		Since (\cite[p. 32, Rule 2]{Annaby-Mansour-Book-2012})
		\[
		{}_{q}L_{s}(D_q^n \varphi)(x)=\left(\frac{s}{1-q}\right)^n {}_{q}L_{s}\varphi(x),
		\]
		provided that $D_q^j \varphi(0^+)=0$ $(j=0,1,\cdots, n-1)$, we have 
		\[
		{}_{q}L_s f(x)={}_{q}L_{s}\big(D_q^n E_{q;\bm{\rho},n-\lambda;(\bm{\omega})}^{-\bm{\gamma}} g\big)(x)
		\]
		if $D_q^j (E_{q;\bm{\rho},n-\lambda;(\bm{\omega})}^{-\bm{\gamma}} g)(0^+)=0$ $(j=0,1,\cdots, n-1)$. Then a formal use of the inversion formula (e.g. \cite[p. 33, Theorem 1.24]{Annaby-Mansour-Book-2012}) leeds us to \eqref{Def-RLtype-1} again.
		\item[(2)] When $k=1$, Eq. \eqref{Def-RLtype-1} becomes
		\begin{equation}\label{Def-RLtype-Remark-1}
			(\mathcal{D}_{q;\rho,\lambda;(\omega)}^{\gamma} f)(x)
			=D_q^n\big(E_{q;\rho,n-\lambda;(\omega)}^{-\gamma} f\big)(x).
		\end{equation} 
		A similar operator was examined in the work of Shaimardan \emph{et al.} \cite[p. 2010, Definition 3.5]{Shaimardan-Karimov-Ruzhansky-Mamanazarov-2025}. Also, the operator \eqref{Def-RLtype-Remark-1} itself is the $q$-analogue of the Prabhakar derivative of Riemann-Liouville-type \cite[p. 26, Eq. (5.9)]{Giusti et al-2020}. 
	\end{itemize}
	
\end{remark}

\begin{definition}\label{Def-Caputotype}
	For a function $f\in AC_q^n[0,a]$, the Caputo-type $q$-fractional derivative ${}^C\mathcal{D}_{q;\bm{\rho},\lambda;(\bm{\omega})}^{\bm{\gamma}}$ is defined by
	\begin{equation}\label{Def-Caputotype-1}
		({}^C\mathcal{D}_{q;\bm{\rho},\lambda;(\bm{\omega})}^{\bm{\gamma}}  f)(x) = \big(E_{q;\bm{\rho},n-\lambda;(\bm{\omega}\circ q^{\bm{\gamma}})}^{-\bm{\gamma}} (D_q^n f)\big)(x),
	\end{equation}
	where $n=[\lambda]+1$.
\end{definition}
\begin{remark}
	~
	\begin{itemize}
		\item[(1)] As shown in (1) of Remark \ref{Def-RLtype-Remark} following Definition \ref{Def-RLtype}, the relation in Definition \ref{Def-Caputotype} can also be obtained by solving a particular integral equation. We can consider the $q$-integral equation: 
		\[
		(E_{q;\bm{\rho},\lambda;(\bm{\omega})}^{\bm{\gamma}} f)(x)=g(x), 
		\]
		where $g$ denotes a known function. Using $q$-Laplace transform and Propositions \ref{Prop-LaplaceT-O} and \ref{Prop-LaplaceT}, we obtain
		\begin{align*}
			{}_q L_{s}f(x)
			&={}_q L_{s}g(x)\cdot \left(\frac{s}{1-q}\right)^{n}\cdot s^{-(n-\lambda)}(1-q)^{n-\lambda}\cdot \prod_{j=1}^{k}\frac{(q^{-\gamma_j}\omega_j q^{\gamma_j}(1-q)^{\rho_j}
				s^{-\rho_j};q)_{\infty}}{(\omega_j q^{\gamma_j}(1-q)^{\rho_j}
				s^{-\rho_j};q)_{\infty}}\\
			&=(1-q)\cdot {}_q L_{s}g(x)\cdot \left(\frac{s}{1-q}\right)^{n}\cdot
			{}_{q}L_{s}z^{n-\lambda-1}E_{\bm{\rho},n-\lambda}^{-\bm{\gamma}}(q;\omega_1 q^{\gamma_1} z^{\rho_1},\cdots,\omega_k q^{\gamma_k} z^{\rho_k}).
		\end{align*}
		If $D_q^j g(0^+)=0$ $(j=0,1,\cdots,n-1)$, then 
		\begin{align*}
			{}_q L_{s}f(x)
			&=(1-q)\cdot {}_q L_{s}(D_q^n g)(x)\cdot
			{}_{q}L_{s}z^{n-\lambda-1}E_{\bm{\rho},n-\lambda}^{-\bm{\gamma}}(q;\omega_1 q^{\gamma_1} z^{\rho_1},\cdots,\omega_k q^{\gamma_k} z^{\rho_k})\\
			&=\big(E_{q;\bm{\rho},n-\lambda;(\bm{\omega}\circ q^{\bm{\gamma}})}^{-\bm{\gamma}} (D_q^n g)\big)(x).
		\end{align*}
		\item[(2)] The operator \eqref{Def-Caputotype-1} is a $q$-analogue of the Prabhakar derivative of Caputo-type (i.e., regularized Prabhakar derivative) \cite[p. 26, Eq. (5.10)]{Giusti et al-2020}. When $\bm{\omega}=\mathbf{0}$, it reduces to the Caputo $q$-fractional derivative \cite[p. 148, Definition 5.2.1]{Annaby-Mansour-Book-2012}:
		\begin{equation}\label{Def-CaputoQFD}
			({}^C D_q^\lambda f)(x)
			=\frac{1}{\Gamma_q(n-\lambda)}\int_0^x x^{n-\lambda-1}(qt/x;q)_{n-\lambda-1}(D^nf)(t)\mathrm{d}_qt
			=I_q^{n-\lambda}D_q^n f(x).
		\end{equation}
	\end{itemize}
\end{remark}

\section{$q$-Fractional differential equations involving \texorpdfstring{$\mathcal{D}_{q;\bm{\rho},\lambda;(\bm{\omega})}^{\bm{\gamma}}$}{E} and \texorpdfstring{${}^C\mathcal{D}_{q;\bm{\rho},\lambda;(\bm{\omega})}^{\bm{\gamma}}$}{E}}\label{Section-4}

\subsection{The eigenvalue problem for  \texorpdfstring{$\mathcal{D}_{q;\bm{\rho},\lambda;(\bm{\omega})}^{\bm{\gamma}}$}{E}}
\label{Subsection-4-1}

Let us consider the Cauchy problem
\begin{equation}\label{EigenProb-1}
	\begin{cases}
		\big(\mathcal{D}_{q;\bm{\rho},\lambda;(\bm{\omega})}^{\bm{\gamma}} y\big)(t)=A y(t),\\
		D_q^{j}\big(E_{q;\bm{\rho},n-\lambda;(\bm{\omega})}^{-\bm{\gamma}} y\big)(0^+)=c_j, j=0,\cdots,n-1.
	\end{cases}
\end{equation}
where $\lambda,\rho_1,\cdots,\rho_k>0$, $n=[\lambda]+1$, $\bm{\gamma},\bm{\rho}\in\mathbb{R}^k$, and $A, c_1,\cdots,c_{n-1}\in\mathbb{R}$. Note that in view of \eqref{Th-Boundedness-Proof-1}, we have
\[
D_q^m (E_{q;\bm{\rho},\lambda;(\bm{\omega})}^{\bm{\gamma}} f)(x) 
=(E_{q;\bm{\rho},\lambda-m;(\bm{\omega})}^{\bm{\gamma}} f)(x),
\]
where $\lambda\geq m$ $(m\in\mathbb{Z}_{\geq0})$. Thus, the initial conditions in \eqref{EigenProb-1} may also be expressed as 
\[
\big(E_{q;\bm{\rho},n-\lambda-j;(\bm{\omega})}^{-\bm{\gamma}} y\big)(0^+)=c_j,\quad j=1,\cdots,n-1.
\]

In the non-$q$ case, the standard approach to the above problem is to use the Laplace transform (see, for example, \cite[p. 28]{Giusti et al-2020} and \cite{Luo-Raina-Chapter-2025}). However, in the $q$-analogue, the use of the $q$-Laplace transform to solve $q$-fractional differential equations has received only little attention in the literature, and has been applied to equations involving relatively simple operators (see \cite[Chapter 9]{Annaby-Mansour-Book-2012}, \cite{Garg-Alha-Chanchlani-2015} and \cite{Garg-Chanchlani-2012}). In fact, there seems a reason for this. In what follows, we shall show that even equations involving more complicated operators such as \eqref{EigenProb-1} can also be solved explicitly using the $q$-Laplace transform.

By means of the $q$-Laplace transform ${}_{q}L_{s}$, we have
\begin{align*}
	{}_{q}L_s \big(\mathcal{D}_{q;\bm{\rho},\lambda;(\bm{\omega})}^{\bm{\gamma}} y\big)(t)
	&={}_{q}L_s \big(D_q^n\big(E_{q;\bm{\rho},n-\lambda;(\bm{\omega})}^{-\bm{\gamma}} y\big)\big)(x)\\
	&=s^n(1-q)^{-n}\cdot {}_{q}L_s \big(E_{q;\bm{\rho},n-\lambda;(\bm{\omega})}^{-\bm{\gamma}} y\big)(x)
	-\sum_{j=0}^{n-1} s^{n-j-1}(1-q)^{j-n} D_q^{j}\big(E_{q;\bm{\rho},n-\lambda;(\bm{\omega})}^{-\bm{\gamma}} f\big)(0^+)\\
	&=s^{\lambda}(1-q)^{-\lambda}\cdot{}_q L_{s}y(x)
	\prod_{j=1}^{k}\frac{(q^{-\gamma_j}\omega_j(1-q)^{\rho_j}
		s^{-\rho_j};q)_{\infty}}{(\omega_j(1-q)^{\rho_j}
		s^{-\rho_j};q)_{\infty}}
	-\sum_{j=0}^{n-1}c_j s^{n-1-j}(1-q)^{j-n}.
\end{align*}
So we can convert the original Cauchy problem into the following equation:
\begin{equation}\label{EigenProb-1-Proof-1}
	\left(s^{\lambda}(1-q)^{-\lambda}
	\Phi(s)
	-A\right) {}_{q}L_s y(t)
	=\sum_{j=0}^{n-1} c_j s^{n-1-j}	(1-q)^{j-n},
\end{equation}
where
\[
\Phi(s):=\prod_{j=1}^{k}\frac{(q^{-\gamma_j}\omega_j(1-q)^{\rho_j}
	s^{-\rho_j};q)_{\infty}}{(\omega_j(1-q)^{\rho_j}
	s^{-\rho_j};q)_{\infty}}.
\]
For convenience, we write $\Psi(s)=\dfrac{1}{\Phi(s)}$. Then from \eqref{EigenProb-1-Proof-1}, we have
\begin{align}\label{EigenProb-1-Proof-2}
	{}_{q}L_s y(t)
	&=\frac{1}{s^{\lambda}(1-q)^{-\lambda}\Phi(s)
		-A}\sum_{j=0}^{n-1}c_j s^{n-1-j}(1-q)^{j-n}\notag\\
	&=\frac{s^{-\lambda} (1-q)^\lambda\Psi(s)}{1-A s^{-\lambda} (1-q)^\lambda\Psi(s)}
	\sum_{j=0}^{n-1} c_j s^{n-1-j}(1-q)^{j-n}.
\end{align}
If we assume that $s\in\mathbb{C}$ $(\Re(s)>0)$ satisfies 
\begin{equation}\label{EigenProb-1-Proof-3}
	|A s^{-\lambda} (1-q)^\lambda\Psi(s)|<1,
\end{equation}
then we can write 
\begin{equation}\label{EigenProb-1-Proof-4}
	\frac{s^{-\lambda} (1-q)^\lambda\Psi(s)}{1-A s^{-\lambda} (1-q)^\lambda\Psi(s)}
	=\sum_{\ell=0}^{\infty}A^\ell s^{-\lambda(\ell+1)}(1-q)^{\lambda(\ell+1)}\Psi^{\ell+1}(s). 
\end{equation}
We will below derive from \eqref{EigenProb-1-Proof-3} a more easily verifiable sufficient condition (see (1) in the Remark below).

Substituting \eqref{EigenProb-1-Proof-4} into \eqref{EigenProb-1-Proof-2}, we get
\begin{align}\label{EigenProb-1-Proof-5}
	{}_{q}L_s y(t)&=\sum_{j=0}^{n-1}c_j
	\sum_{\ell=0}^{\infty}A^{\ell}
	s^{-\lambda(\ell+1)+n-j-1}
	(1-q)^{\lambda(\ell+1)-n+j}\Psi^{\ell+1}(s)\notag\\
	&=\sum_{j=0}^{n-1}c_j
	\sum_{\ell=0}^{\infty}A^{\ell}s^{-\lambda(\ell+1)+n-j-1}
	(1-q)^{\lambda(\ell+1)-n+j} \prod_{j=1}^{k}\left(\frac{(\omega_j(1-q)^{\rho_j}
		s^{-\rho_j};q)_{\infty}}{(q^{-\gamma_j}\omega_j(1-q)^{\rho_j}
		s^{-\rho_j};q)_{\infty}}\right)^{\ell+1}.
\end{align}

Next, our task is to express the summand in \eqref{EigenProb-1-Proof-5} as the $q$‑Laplace transform of a multivariate $q$‑Mittag‑Leffler function. To this end, we introduce a useful notation. Let $\oplus$ denote the \emph{concatenation} of vectors. For example, if $\bm{\gamma}=(\gamma_1,\cdots,\gamma_k)$ and $\bm{\rho}=(\rho_1,\cdots,\rho_k)$, then 
$\bm{\gamma}\oplus\bm{\rho}=
(\gamma_1,\cdots,\gamma_k,\rho_1,\cdots,\rho_k)$. By using Proposition \ref{Prop-LaplaceT}, we obtain 
\begin{align*}
	{}_{q}L_s y(t)
	&=\sum_{j=0}^{n-1}c_j
	\sum_{\ell=0}^{\infty}A^{\ell}
	{}_{q}L_s t^{\lambda(\ell+1)-n+j}\\
	&\hspace{0.5cm}\cdot E_{\bm{\rho}_1\oplus\cdots\oplus\bm{\rho}_k, \lambda(\ell+1)-n+j+1}^{\bm{\gamma}_1\oplus\cdots\oplus\bm{\gamma}_k}
	\big(q;\underbrace{\omega_1 q^{-\gamma_1}t^{\rho_1},\cdots,\omega_1 q^{-\gamma_1}t^{\rho_1}}_{\text{$\ell+1$ terms}},\cdots,\underbrace{\omega_k q^{-\gamma_k}t^{\rho_k},\cdots,\omega_k q^{-\gamma_k}t^{\rho_k}}_{\text{$\ell+1$ terms}}\big),
\end{align*}
where we have used the convention that
$\bm{\gamma}_j=(\gamma_j,\cdots,\gamma_j)$ and $\bm{\rho}_j=(\rho_j,\cdots,\rho_j)$ $(j=1,\cdots,k)$. 

Finally, by taking the inverse $q$-Laplace transform, the solution to Cauchy problem \eqref{EigenProb-1} is then given by 
\begin{align*}
	y(t)
	&=\sum_{j=0}^{n-1}c_j
	\sum_{\ell=0}^{\infty}A^{\ell}
	t^{\lambda(\ell+1)-n+j}\\
	&\hspace{0.5cm}\cdot E_{\bm{\rho}_1\oplus\cdots\oplus\bm{\rho}_k, \lambda(\ell+1)-n+j+1}^{\bm{\gamma}_1\oplus\cdots\oplus\bm{\gamma}_k}
	\big(q;\underbrace{\omega_1 q^{-\gamma_1}t^{\rho_1},\cdots,\omega_1 q^{-\gamma_1}t^{\rho_1}}_{\text{$\ell+1$ terms}},\cdots,\underbrace{\omega_k q^{-\gamma_k}t^{\rho_k},\cdots,\omega_k q^{-\gamma_k}t^{\rho_k}}_{\text{$\ell+1$ terms}}\big).
\end{align*}

\begin{remark}\label{Subsection-4-1-Remark}
~
\begin{itemize}
	\item[(1)] Since $\displaystyle 
	\lim_{z\rightarrow0}\frac{(q^{-\gamma}z;q)_{\infty}}{(z;q)_{\infty}}=1$, we have $|\Phi(s)|\rightarrow1$ as $|s|\rightarrow\infty$ $(\Re(s)>0)$, which means that there exists a positive number $M$ so that $|\Phi(s)|>\dfrac{1}{2}$ $(|s|>M, \Re(s)>0)$, or equivalently $|\Psi(s)|<2$  $(|s|>M, \Re(s)>0)$. Thus 
	\[
	|s|>\max\left\{M,(2|A|)^{1/\lambda}(1-q)\right\}
	\]
	is a sufficient condition for \eqref{EigenProb-1-Proof-3} to hold.
	\item[(2)] When $k=1$ in \eqref{EigenProb-1}, the original problem has the form:
	\begin{equation}\label{EigenProb-1-Remark-1}
		\begin{cases}
			\big(\mathcal{D}_{q;\rho,\lambda;(\omega)}^{\gamma} y\big)(t)=A y(t),\\
			D_q^{j}\big(E_{q;\rho,n-\lambda;(\omega)}^{-\gamma} y\big)(0^+)=c_j, j=0,\cdots,n-1,
		\end{cases}
	\end{equation}
	and its solution is given by
	\begin{align*}
		y(t)
		&=\sum_{j=0}^{n-1}c_j
		\sum_{\ell=0}^{\infty}A^{\ell}
		t^{\lambda(\ell+1)-n+j}E_{(\rho,\cdots,\rho), \lambda(\ell+1)-n+j+1}^{(\gamma,\cdots,\gamma)}
		\big(q;\underbrace{\omega q^{-\gamma}t^{\rho},\cdots,\omega q^{-\gamma}t^{\rho}}_{\text{$\ell+1$ terms}}\big),
	\end{align*}
	where
	\begin{align}\label{EigenProb-1-Remark-2}
		&E_{(\rho,\cdots,\rho), \lambda(\ell+1)-n+j+1}^{(\gamma,\cdots,\gamma)}
		\big(q;\underbrace{\omega q^{-\gamma}t^{\rho},\cdots,\omega q^{-\gamma}t^{\rho}}_{\text{$\ell+1$ terms}}\big)\notag\\
		&\hspace{0.5cm}=\sum_{n_1,\cdots,n_{\ell+1}=0}^{\infty}
		\frac{(q^{\gamma};q)_{n_1}}{(q;q)_{n_1}}\cdots\frac{(q^{\gamma};q)_{n_{\ell+1}}}{(q;q)_{n_{\ell+1}}} 
		\frac{(\omega q^{-\gamma}t^{\rho})^{(n_1+\cdots+n_{\ell+1})}}{\Gamma_q(\lambda(\ell+1)-n+j+1 + \rho(n_1+\cdots+n_{\ell+1}))}.
	\end{align}
	First, we point out that \eqref{EigenProb-1-Remark-1} is a $q$-analogue of the eigenvalue problem for the Prabhakar derivative of Riemann-Liouville-type \cite[p. 28, Eq. (5.17)]{Giusti et al-2020}. Second, we want to emphasize that the $(\ell+1)$-fold series given by \eqref{EigenProb-1-Remark-2} can no longer be simplified because when we set $N=n_1+\cdots+n_{\ell+1}$, it becomes 
	\begin{align}\label{EigenProb-1-Remark-3}
		&E_{(\rho,\cdots,\rho), \lambda(\ell+1)-n+j+1}^{(\gamma,\cdots,\gamma)}
		\big(q;\underbrace{\omega q^{-\gamma}t^{\rho},\cdots,\omega q^{-\gamma}t^{\rho}}_{\text{$\ell+1$ terms}}\big)\notag\\
		&\hspace{0.5cm}=\sum_{N=0}\left(\sum_{n_1+\cdots+n_{\ell+1}=N}
		\frac{(q^{\gamma};q)_{n_1}}{(q;q)_{n_1}}\cdots\frac{(q^{\gamma};q)_{n_{\ell+1}}}{(q;q)_{n_{\ell+1}}} \right) 
		\frac{(\omega q^{-\gamma}t^{\rho})^{N}}{\Gamma_q(\lambda(\ell+1)-n+j+1 + \rho N)}.
	\end{align}
	but among them 
	\begin{equation}\label{EigenProb-1-Remark-4}
		\sum_{n_1+\cdots+n_{\ell+1}=N}
		\frac{(q^{\gamma};q)_{n_1}}{(q;q)_{n_1}}\cdots\frac{(q^{\gamma};q)_{n_{\ell+1}}}{(q;q)_{n_{\ell+1}}}
	\end{equation}
	has \emph{no} known summation formula for arbitrary $\ell\in\mathbb{Z}_{\geq2}$. (Compared with the well‑known $q$‑multinomial theorem \cite[p. 25]{Gasper-Rahman-Book-2004}, the summand in formula \eqref{EigenProb-1-Remark-4} is evidently missing a crucial $q$‑factor.) In the non-$q$ case, we indeed have a formula given in \cite[p. 2082, Eq. (34)]{Luo-Raina-2014}
	\begin{equation}\label{EigenProb-1-Remark-5}
		\sum_{m_1+\cdots+m_k=m}\frac{(\gamma)_{m_1}}{m_1!}\cdots\frac{(\gamma)_{m_k}}{m_k!}=\frac{(\gamma k)_{m}}{m!}.
	\end{equation}
	If we denote the limiting expression of \eqref{EigenProb-1-Remark-2} as $q\rightarrow 1^{-}$ by $E_{(\rho,\cdots,\rho), \lambda(\ell+1)-n+j+1}^{(\gamma,\cdots,\gamma)}
	\big(\underbrace{\omega t^{\rho},\cdots,\omega t^{\rho}}_{\text{$\ell+1$ terms}}\big)$, then from \eqref{EigenProb-1-Remark-5}, we immediately obtain:
	\begin{align*}
		E_{(\rho,\cdots,\rho), \lambda(\ell+1)-n+j+1}^{(\gamma,\cdots,\gamma)}
		\big(\underbrace{\omega t^{\rho},\cdots,\omega t^{\rho}}_{\text{$\ell+1$ terms}}\big)
		&=\sum_{N=0}\frac{(\gamma(\ell+1))_{N}}{N!}
		\frac{(\omega t^{\rho})^{N}}{\Gamma(\lambda(\ell+1)-n+j+1 + \rho N)}\\
		&=E_{\rho,\lambda(\ell+1)-n+j+1}^{\gamma(\ell+1)}(\omega t^\rho),
	\end{align*}
	where $E_{\rho,\lambda}^{\gamma}(z)$ denotes the Prabhakar function  
	\cite[p. 17, Eq. (4.1)]{Giusti et al-2020}. This brings us back to the solution to the eigenvalue problem of the Prabhakar derivative of the Riemann-Liouville type \cite[p. 28, Eq. (5.18)]{Giusti et al-2020}. Finally, we mention that \eqref{EigenProb-1-Remark-4} can be viewed as a special value of the $q$-Lagrange polynomials introduced by Alt\i n \emph{et al.} \cite[p. 1132, Eq. (2.3)]{Altin-Erkus-Tasdelen-2006}.
\end{itemize}
\end{remark}
	
\subsection{The eigenvalue problem for \texorpdfstring{${}^C\mathcal{D}_{q;\bm{\rho},\lambda;(\bm{\omega})}^{\bm{\gamma}}$}{E}}

Let us consider the Cauchy problem
\begin{equation}\label{EigenProb-2}
	\begin{cases}
		\big({}^C\mathcal{D}_{q;\bm{\rho},\lambda;(\bm{\omega})}^{\bm{\gamma}} y\big)(t)=A y(t),\\
		D_q^j y(0^+)=c_j, j=0,\cdots,n-1.
	\end{cases}
\end{equation}
By means of the $q$-Laplace transform, we have
\begin{align*}
	{}_{q}L_s\big({}^C\mathcal{D}_{q;\bm{\rho},\lambda;(\bm{\omega})}^{\bm{\gamma}} y\big)(t)
	&={}_{q}L_s \big(E_{q;\bm{\rho},n-\lambda;(\bm{\omega}\circ q^{\bm{\gamma}})}^{-\bm{\gamma}} (D_q^n y)\big)(x)\\
	&={}_q L_{s}(D_q^n y)(x)\cdot s^{-(n-\lambda)}(1-q)^{n-\lambda}\cdot 
	\prod_{j=1}^{k}\frac{(\omega_j(1-q)^{\rho_j}
		s^{-\rho_j};q)_{\infty}}{(q^{\gamma_j}\omega_j(1-q)^{\rho_j}
		s^{-\rho_j};q)_{\infty}}\\
	&=\left({}_{q}L_s y(x)-\sum_{j=0}^{n-1} c_j s^{-j-1}(1-q)^j\right)\\
	&\hspace{1cm}\cdot s^{\lambda}(1-q)^{-\lambda} 
	\prod_{j=1}^{k}\frac{(\omega_j(1-q)^{\rho_j}
		s^{-\rho_j};q)_{\infty}}{(q^{\gamma_j}\omega_j(1-q)^{\rho_j}
		s^{-\rho_j};q)_{\infty}}, 
\end{align*}
which implies
\begin{equation}\label{EigenProb-2-Proof-1}
	\big(s^\lambda(1-q)^{-\lambda}\Xi(s)-A\big){}_{q}L_s y(x)
	=\Xi(s)\sum_{j=0}^{n-1} c_j s^{\lambda-j-1}(1-q)^{-\lambda+j},
\end{equation}
where
\[
\Xi(s):=\prod_{j=1}^{k}\frac{(\omega_j(1-q)^{\rho_j}
	s^{-\rho_j};q)_{\infty}}{(q^{\gamma_j}\omega_j(1-q)^{\rho_j}
	s^{-\rho_j};q)_{\infty}}.
\]
For convenience, we write $\Theta(s)=\dfrac{1}{\Xi(s)}$. By the same argument as in subsection \ref{Subsection-4-1}, we obtain from \eqref{EigenProb-2-Proof-1} that
\begin{align}\label{EigenProb-2-Proof-2}
	{}_{q}L_s y(x)
	&=\frac{\Xi(s)}{s^\lambda(1-q)^{-\lambda}\Xi(s)-A}\sum_{j=0}^{n-1} c_j s^{\lambda-j-1}(1-q)^{-\lambda+j}\notag\\
	&=\frac{1}{1-As^{-\lambda}(1-q)^{\lambda}\Theta(s)}\sum_{j=0}^{n-1} c_j s^{-j-1}(1-q)^{j}\notag\\
	&=\sum_{j=0}^{n-1} c_j 
	\sum_{\ell=0}^{\infty}A^\ell s^{-\lambda \ell-j-1}(1-q)^{\lambda \ell+j}\Theta^\ell (s), 
\end{align}
provided that $|As^{-\lambda}(1-q)^{\lambda}\Theta(s)|<1$. Using Proposition \ref{Prop-LaplaceT}, we have
\begin{align}\label{EigenProb-2-Proof-3}
	&s^{-\lambda \ell-j-1}(1-q)^{\lambda \ell+j}\Theta^\ell (s)\notag\\
	&\hspace{1cm}=s^{-\lambda \ell-j-1}(1-q)^{\lambda \ell+j}
	\prod_{j=1}^{k}\left(\frac{(q^{\gamma_j}\omega_j(1-q)^{\rho_j}
		s^{-\rho_j};q)_{\infty}}{(\omega_j(1-q)^{\rho_j}
		s^{-\rho_j};q)_{\infty}}\right)^\ell \notag\\
	&\hspace{1cm}={}_{q}L_{s} x^{\lambda \ell+j}
	E_{\bm{\rho}_1\oplus\cdots\oplus\bm{\rho}_k, \lambda\ell+j+1}^{\bm{\gamma}_1\oplus\cdots\oplus\bm{\gamma}_k}
	\big(q;\underbrace{\omega_1 t^{\rho_1},\cdots,\omega_1 t^{\rho_1}}_{\text{$\ell+1$ terms}},\cdots,\underbrace{\omega_k t^{\rho_k},\cdots,\omega_k t^{\rho_k}}_{\text{$\ell+1$ terms}}\big),
\end{align}
where $\bm{\gamma}_j=(\gamma_j,\cdots,\gamma_j)$ and $\bm{\rho}_j=(\rho_j,\cdots,\rho_j)$ $(j=1,\cdots,k)$, and the operation $\oplus$ is defined in subsection \ref{Subsection-4-1}.

Substituting \eqref{EigenProb-2-Proof-3} into \eqref{EigenProb-2-Proof-2} and then taking the inverse $q$-Laplace transform, we obtain the solution to Cauchy problem \eqref{EigenProb-2} as follows: 
\begin{align*} 
	y(s)&=\sum_{j=0}^{n-1} c_j 
	\sum_{\ell=0}^{\infty}A^\ell x^{\lambda \ell+j}
	E_{\bm{\rho}_1\oplus\cdots\oplus\bm{\rho}_k, \lambda\ell+j+1}^{\bm{\gamma}_1\oplus\cdots\oplus\bm{\gamma}_k}
	\big(q;\underbrace{\omega_1 t^{\rho_1},\cdots,\omega_1 t^{\rho_1}}_{\text{$\ell+1$ terms}},\cdots,\underbrace{\omega_k t^{\rho_k},\cdots,\omega_k t^{\rho_k}}_{\text{$\ell+1$ terms}}\big).
\end{align*}

\begin{remark}
	When $k=1$ in \eqref{EigenProb-2}, the original problem has the form: 
	\begin{equation}\label{EigenProb-2-Remark-1}
		\begin{cases}
			\big({}^C\mathcal{D}_{q;\rho,\lambda;(\omega)}^{\gamma} y\big)(t)=A y(t),\\
			D_q^j y(0^+)=c_j, j=0,\cdots,n-1,
		\end{cases}
	\end{equation}
	and its solution is given by
	\[
	y(s)=\sum_{j=0}^{n-1} c_j 
	\sum_{\ell=0}^{\infty}A^\ell x^{\lambda \ell+j}
	E_{(\rho,\cdots,\rho), \lambda\ell+j+1}^{(\gamma,\cdots,\gamma)}
	\big(q;\underbrace{\omega t^{\rho},\cdots,\omega t^{\rho}}_{\text{$\ell+1$ terms}}\big),
	\]
	where
	\begin{align*}
		&E_{(\rho,\cdots,\rho), \lambda\ell+j+1}^{(\gamma,\cdots,\gamma)}
		\big(q;\underbrace{\omega t^{\rho},\cdots,\omega t^{\rho}}_{\text{$\ell+1$ terms}}\big)\\
		&\hspace{0.5cm}=\sum_{n_1,\cdots,n_{\ell+1}=0}^{\infty}
		\frac{(q^{\gamma};q)_{n_1}}{(q;q)_{n_1}}\cdots\frac{(q^{\gamma};q)_{n_{\ell+1}}}{(q;q)_{n_{\ell+1}}} 
		\frac{(\omega t^{\rho})^{(n_1+\cdots+n_{\ell+1})}}{\Gamma_q(\lambda\ell+j+1 + \rho(n_1+\cdots+n_{\ell+1}))}.
	\end{align*}
	Using our similar analysis as mentioned in Remark \ref{Subsection-4-1-Remark} of subsection \ref{Subsection-4-1}, one can readily conclude that \eqref{EigenProb-2-Remark-1} is a $q$-analogue of the eigenvalue problem for the Prabhakar derivative of Caputo-type \cite[p. 28, Eq. (5.14)]{Giusti et al-2020}. 
\end{remark}

The $q$-Laplace transform method employed above also applies to the non-homogeneous versions of problems \eqref{EigenProb-1} and \eqref{EigenProb-2}, respectively. We omit here further details about such considerations. In our forthcoming work, we propose to consider two classes of general Cauchy initial value problems.

\section{Concluding remarks}\label{Section-5}

\begin{itemize}
	\item[(1)] An interesting remark attributed in the paper \cite{Mainardi-Gorenflo-2007}, states that the  “Mittag-Leffler functions are regarded as the queen function of fractional calculus”. Very useful and interesting research work continues to motivate researchers to consider problems of interconnections and applications with other related branches, see \cite{Gorenflo-Kilbas-Mainardi-Rogosin-Book-2020}. Not only have univariate Mittag-Leffler functions been continuously studied in depth (see \cite{Abdelhakim-2025}, \cite{Bie et al-2025}, \cite{Garrappa-Gerhold-Popolizio-Simon-2025}), but various multivariate generalizations have also emerged in the works in \cite{Abilassan-Restrepo-Suragan-2023} and  \cite{Giordano-2025}. The multivariate $q$-Mittag-Leffler function $E_{\bm{\rho},\lambda}^{\bm{\gamma}}(q;\mathbf{z})$ introduced in this paper has been discussed in some detail in Section \ref{Section-2}.  There seems scope for further consideration and investigations for other interesting problems. We specificaly like to suggest here one research direction deserving further attention: Find a (multiple) Mellin-Barnes contour integral representation for $E_{\bm{\rho},\lambda}^{\bm{\gamma}}(q;\mathbf{z})$, and thereby investigate its asymptotic behaviour.  
	\item[(2)] Various analytical properties of fractional integral operators including boundedness and compactness on various function spaces remain an important and active research proposal. In this paper, we only study the boundedness of the operator $E_{q;\bm{\rho},\lambda;(\bm{\omega})}^{\bm{\gamma}}$ on $L_q^p[0,a]$ $(p\geq1)$. However, as suggested by the work of Ledesma \emph{et al.} \cite{Ledesma et al-2026} (and the reference therein), studying the boundedness or compactness of $E_{q;\bm{\rho},\lambda;(\bm{\omega})}^{\bm{\gamma}}$ on other spaces may be another meaningful and important proposal (or a direction) for future research. 
	
\end{itemize}

\section*{Funding}

No funding was received for conducting this study.

\section*{Contributions}

All authors contributed equally to the preparation of this paper.

\section*{Conflicts of interest}

The authors declare that there is no conflict of interest.

\section*{Data Availability}

This manuscript has no associated data.

\section*{Acknowledgements}

We are grateful to Peng-Cheng Hang for his careful reading of the initial manuscript and providing valuable suggestions.


\begin{thebibliography}{99}
	\footnotesize
	\setlength{\itemsep}{-2pt}
	
	\bibitem{Abilassan-Restrepo-Suragan-2023}
	A. Abilassan, J.E. Restrepo, D. Suragan, 
	On a variant of multivariate Mittag-Leffler's function arising in the Laplace transform method, 
	Integral Transforms Spec. Funct. 34 (3) (2023), 244--260.
	\url{https://doi.org/10.1080/10652469.2022.2111420}
	
	\bibitem{Abdelhakim-2025}
	A.A. Abdelhakim, 
	Mittag-Leffler functions as eigenfunctions of the Hilbert transform, 
	Math. Methods Appl. Sci. 48 (8) (2025), 9340--9343.
	\url{https://doi.org/10.1002/mma.10776}
	
	\bibitem{Abdeljawad et al-2011}
	T. Abdeljawad, D. Baleanu, 
	Caputo $q$-fractional initial value problems and a $q$-analogue Mittag-Leffler function, 
	Commun. Nonlinear Sci. Numer. Simul. 16 (12) (2011), 4682--4688.
	\url{https://doi.org/10.1016/j.cnsns.2011.01.026}
	
	\bibitem{Agarwal-1969}
	R.P. Agarwal,
	Certain fractional $q$-integrals and $q$-derivatives, 
	Proc. Camb. Phil. Soc. 66 (1969), 365--370. 
	\url{https://doi.org/10.1017/S0305004100045060}
	
	\bibitem{AlSalam-1966}
	W.A. Al-Salam, 
	Some fractional $q$-integrals and $q$-derivatives, 
	Proc. Edinb. Math. Soc. 2 (15) (1966), 135--140.
	\url{https://doi.org/10.1017/S0013091500011469}
	
	\bibitem{AlTowailb-2019}
	M.A. Al-Towailb,
	The solution of certain triple $q$-integral equations in fractional $q$-calculus approach,
	Arab J. Math. Sci. 25 (1) (2019), 17--28.
	\url{https://doi.org/10.1016/j.ajmsc.2018.10.001} 
	
	\bibitem{Annaby-Mansour-Book-2012}
	M.H. Annaby, Z.S. Mansour,
	$q$-Fractional Calculus and Equations,
	Springer, 2012.
	
	\bibitem{Altin-Erkus-Tasdelen-2006}
	A. Alt\i n, E. Erku\c{s}, F. Ta\c{s}delen, 
	The $q$-Lagrange polynomials in several variables, 
	Taiwanese J. Math. 10 (5) (2006), 1131--1137.
	\url{https://doi.org/10.11650/twjm/1500557293}
	
	\bibitem{Bie et al-2025}
	H.D. Bie, P. Lian, F. Maes, 
	Bounds for the kernel of the $(\kappa,a)$-generalized Fourier transform, 
	J. Funct. Anal. 288 (4) (2025), Article ID 110755, 29 pp. 
	\url{https://doi.org/10.1016/j.jfa.2024.110755}
	
	\bibitem{Bouzeffour-2026}
	F. Bouzeffour, 
	General fractional $q$-integrals and 
	$q$-derivatives: a Sonin-Luchko $q$-kernel approach, 
	Arab. J. Math. 15 (2) (2026), 599--610.
	\url{https://doi.org/10.1007/s40065-025-00545-9}
	
	\bibitem{Cardoso-Petronilho-2015}
	J.L. Cardoso, J. Petronilho,
	Variations around Jackson's quantum operator,
	Methods Appl. Anal. 22 (4) (2015), 343--358.
	\url{https://dx.doi.org/10.4310/MAA.2015.v22.n4.a1}
	
	\bibitem{Castro-2025}
	L.P. Castro, 
	On the solutions to a Riemann-Liouville fractional $q$-derivative boundary value problem, 
	Fract. Calc. Appl. Anal. 28 (5) (2025), 2302--2332.
	\url{https://doi.org/10.1007/s13540-025-00441-1}
	
	\bibitem{Castro-2026}
	L.P. Castro, 
	Positive solutions for a class of two-term fractional $q$-derivative boundary value problems with an integral condition,
	Differ. Equ. Dyn. Syst. 34 (2) (2026), 499--524.
	\url{https://doi.org/10.1007/s12591-025-00751-z}
	
	\bibitem{Cinque-2021}
	F. Cinque, 
	On the sum of independent generalized Mittag-Leffler random variables and the related fractional processes, 
	Stochastic Anal. Appl. 40 (1) (2022), 103--117.
	\url{https://doi.org/10.1080/07362994.2021.1890120}
	
	\bibitem{Cinque-Orsingher-2024}
	F. Cinque, E. Orsingher, 
	Analysis of fractional Cauchy problems with some probabilistic applications, 
	J. Math. Anal. Appl. 536 (1) (2024), Article ID 128188, 23 pp.
	\url{https://doi.org/10.1016/j.jmaa.2024.128188}
	
	\bibitem{Ernst-2023}
	T. Ernst, 
	On quadruple $q$-hypergeometric functions and diverse generalizations to $n$ variables in the spirit of Exton, 
	Bull. Iran. Math. Soc. 49 (5) (2023), Paper No. 73, 23 pp.
	\url{https://doi.org/10.1007/s41980-023-00809-0}
	
	\bibitem{Garg-Alha-Chanchlani-2015}
	M. Garg, S. Alha, L. Chanchlani, 
	On generalized composite fractional $q$-derivative, 
	Matematiche 70 (2) (2015), 123--134.
	
	\bibitem{Garg-Chanchlani-2012}
	M. Garg, L. Chanchlani, 
	On fractional $q$-kinetic equation, 
	Mat. Bilt. 36 (2012), 33--46.
	
	\bibitem{Garrappa-Gerhold-Popolizio-Simon-2025}
	R. Garrappa, S. Gerhold, M. Popolizio, T. Simon, 
	On some inequalities for the two-parameter Mittag-Leffler function in the complex plane, 
	J. Math. Anal. Appl. 551 (2025), 129588.
	\url{https://doi.org/10.1016/j.jmaa.2025.129588}
	
	\bibitem{Gasper-Rahman-Book-2004}
	G. Gasper, M. Rahman,
	Basic Hypergeometric Series,
	Cambridge University Press, Cambridge, 2004.
	
	\bibitem{Giordano-2025}
	C. Giordano, 
	Generalized Mittag-Leffler functions, Borel-Laplace multitransforms, and quantum differential equations of Fano varieties, 2025. 
	\url{https://arxiv.org/abs/2506.18033}.
	
	\bibitem{Giusti et al-2020}
	A. Giusti, I. Colombaro, R. Garra, R. Garrappa, F. Polito, M. Popolizio, F. Mainardi,
	A practical guide to Prabhakar fractional calculus,
	Fract. Calc. Appl. Anal., 23 (1) (2020), 9--54.
	\url{https://doi.org/10.1515/fca-2020-0002} 
	
	\bibitem{Gorenflo-Kilbas-Mainardi-Rogosin-Book-2020}
	R. Gorenflo, A.A. Kilbas, F. Mainardi, S.V. Rogosin,
	Mittag-Leffler functions, Related Topics and Applications,
	Springer, 2020.

	\bibitem{Gonzalez-Bayron-2026}
	M.P. Gonzalez, M.-R. Bayron, 
	A $q$-Caputo fractional generalization of Tsallis entropy, 
	Int. J. Theor. Phys. 65 (7) (2026), Paper No. 189, 21 pp.
	\url{https://doi.org/10.1007/s10773-026-06389-z}
	
	\bibitem{Kac-Cheung-Book-2002}
	V. Kac, P. Cheung, Quantum calculus, Universitext, Springer, New York, NY, 2002.
	
	\bibitem{Kilbas-Srivastava-Trujillo-Book-2006}
	A.A. Kilbas, H.M. Srivastava, J.J. Trujillo,
	Theory and Applications of Fractional Differential Equations,
	Elsevier, 2006.
	
	\bibitem{Ledesma et al-2026}
	C.E.T. Ledesma, Y. Luchko, G.A.M. Cruz, W.Z. Vera, 
	Some boundedness results for the Prabhakar integral operator, 
	Mediterr. J. Math. 23 (2026), Paper No. 113, 22 pp. 
	\url{https://doi.org/10.1007/s00009-026-03103-5}
	
	\bibitem{Luo-Zhou-2026}
	M.-J. Luo, X.-L. Zhou, Babenko's method for $q$-fractional integral equations: a systematic framework, 
	Math. Methods Appl. Sci. 49 (14) (2026), 15674--15695.
	\url{https://doi.org/10.1002/mma.70845}
	
	\bibitem{Luo-Raina-Zhou-2025}
	M.-J. Luo, R.K. Raina, X.-L. Zhou,
	Applications of Babenko's method to a class of integral equations,
	Fract. Calc. Appl. Anal. 28 (6) (2025), 2748--2773.
	\url{https://doi.org/10.1007/s13540-025-00454-w}
	
	\bibitem{Luo-Raina-2014}
	M.-J. Luo, R.K. Raina, 
	On certain classes of fractional kinetic equations, 
	Filomat 28 (10) (2014), 2077--2090.
	\url{https://doi.org/10.2298/FIL1410077L} 
	
	\bibitem{Luo-Raina-Chapter-2025}
	M.-J. Luo, R.K. Raina, Laplace transformation of fractional integrals and derivatives, in \emph{The Fundamentals of Fractional Calculus}, eds. D.K. Singh and M. Yavuz, CRC Press/Apple Academic Press, Boca Raton, FL, 2025. 
	
	\bibitem{Mainardi-Gorenflo-2007}
	F. Mainardi, R. Gorenflo, 
	Time-fractional derivatives in relaxation processes: a tutorial survey,  
	Fract. Calc. Appl. Anal. 10 (3) (2007), 269--308.
	
	\bibitem{Maligaranda-Oinarov-Persson-2014}
	L. Maligranda, R. Oinarov, L.-E. Persson, 
	On Hardy $q$-inequalities, 
	Czech. Math. J. 64 (3) (2014), 659--682.
	\url{https://doi.org/10.1007/s10587-014-0125-6} 
	
	\bibitem{Mamanazarov-2024}
	A. Mamanazarov, The Prabhakar Fractional $q$-Integral and $q$-Differential Operators, in \emph{Extended Abstracts 2021/2022}, ed. M. Ruzhansky and K. van Bockstal, Birkh\"{a}user, 2024.
	\url{https://doi.org/10.1007/978-3-031-42539-4_11}
	
	\bibitem{Ozarslan-Fernandez-2022}
	M.A. \"{O}zarslan, A. Fernandez, 
	On the fractional calculus of multivariate Mittag-Leffler functions, 
	Int. J. Comput. Math. 99 (2) (2022), 247--273.
	
	\bibitem{Purohit-Kalla-2011}
	S.D. Purohit, S.L. Kalla, 
	A generalization of $q$-Mittag-Leffler function, 
	Mat. Bilt. 35 (2011), 15--26.
	
	\bibitem{Sakhno-Storozhuk-2026}
	L. Sakhno, A. Storozhuk, 
	On some extensions of generalized counting processes, 2026. 
	\url{https://arxiv.org/abs/2604.00882}
	
	\bibitem{Saxena-Kalla-Saxena-2011}
	R.K. Saxena, S.L. Kalla, R. Saxena, 
	Multivariate analogue of generalized Mittag-Leffler function, 
	Integral Transforms Spec. Funct. 22 (2011), 533--548.
	\url{https://doi.org/10.1080/10652469.2010.533474}
	
	\bibitem{Shaimardan-Tokmagambetov-2021}
	S. Shaimardan, N.S. Tokmagambetov, 
	On the solutions of some fractional $q$-differential equations with the Riemann-Liouville fractional $q$-derivative, 
	Bull. Karaganda Univ., Math. Ser. 104 (4) (2021), 130--141.
	
	\bibitem{Shaimardan-Persson-Tokmagambetov-2020}
	S. Shaimardan, L.E. Persson, N.S. Tokmagambetov,
	Existence and uniqueness of some Cauchy type problems in fractional $q$-difference calculus,
	Filomat 34 (13) (2020), 4429--4444.
	\url{https://doi.org/10.2298/FIL2013429S} 
	
	\bibitem{Shaimardan-Karimov-Ruzhansky-Mamanazarov-2025}
	S. Shaimardan, E. Karimov, M. Ruzhansky, A. Mamanazarov, 
	The Prabhakar fractional $q$-integral and $q$-differential operators, and their properties, 
	Filomat 39 (6) (2025), 2003--2016.
	\url{https://doi.org/10.2298/FIL2506003S}
	
	\bibitem{Sheng-Zhang-2022a}
	Y. Sheng, T. Zhang, 
	Some results on the $q$-calculus and fractional $q$-differential equations, 
	Mathematics, 10 (2022), 64.
	\url{https://doi.org/10.3390/math10010064}
	
	\bibitem{Srivastava-Karlsson-Book-1985}
	H.M. Srivastava, P.W. Karlsson, 
	Multiple Gaussian Hypergeometric Series, 
	Ellis Horwood Ltd., Chichester, 1985.
	
	\bibitem{Stankovc-Rajkovic-Marinkovic-2009}
	M.S. Stankovi\'{c}, P.M. Rajkovi\'{c}, S.D. Marinkovi\'{c}, 
	On fractional $q$-integrals, $q$-derivatives and corresponding equations, in Application of mathematics in technical and natural sciences, 300--308, AIP Conf. Proc., 1186 Amer. Inst. Phys., Melville, NY.
	
	\bibitem{Tsallis-1988}
	C. Tsallis, 
	Possible generalization of Boltzmann--Gibbs statistics, 
	J. Stat. Phys. 52 (1-2) (1988), 479--487.
	\url{https://doi.org/10.1007/BF01016429}
	
	\bibitem{Zhang-Guo-2020}
	T. Zhang, Q. Guo, 
	The solution theory of the nonlinear $q$-fractional differential equations,
	Appl. Math. Lett. 104 (2020), 106282.
	\url{https://doi.org/10.1016/j.aml.2020.106282}
	
	\bibitem{Zhang-Wang-2021}
	T. Zhang, Y. Wang, 
	The unique existence of solution in the $q$-integrable space for the nonlinear $q$-fractional differential equations, 
	Fractals 29 (3) (2021), 2150050.
	\url{https://doi.org/10.1142/S0218348X2150050X} 
\end{thebibliography}
\end{document}